\documentclass[10pt,reqno]{amsart} % please use amsart at 11pt

\usepackage{amssymb,latexsym}
\usepackage{cite} % to get Refs. [1,2,3] typeset as [1--3] automatically

\usepackage[height=190mm,width=130mm]{geometry} % this is the journal
\theoremstyle{plain}

\newtheorem{lemma}{Lemma}
\newtheorem{corollary}{Corollary}
\newtheorem{proposition}{Proposition}

\theoremstyle{definition}
\newtheorem{definition}{Definition}

\theoremstyle{remark}
\newtheorem{remark}{Remark}

\newcommand{\Z}{\mathbb{Z}}
\newcommand{\C}{\mathbb{C}}

\numberwithin{equation}{section} % to get equations numbered
\newcommand{\N}{\ensuremath{\mathbb {N}}}

\newcommand{\F}{\ensuremath{\mathcal{F}}}

\newcommand{\g}{\ensuremath{\Gamma}}

\newcommand{\ps}{{\raise 1pt\hbox{\tiny (}}}

\newcommand{\pss}{{\raise 1pt\hbox{\tiny [}}}
\newcommand{\pdd}{{\raise 1pt\hbox{\tiny ]}}}
\newcommand{\pd}{{\raise 1pt\hbox{\tiny )}}}

\newcommand{\bs}{{\raise 1pt\hbox{\tiny [}}}
\newcommand{\bd}{{\raise 1pt\hbox{\tiny ]}}}

\def\cross{\mathinner{\mathrel{\raise0.8pt\hbox{$\scriptstyle>$}}
                 \joinrel\mathrel\triangleleft}}

\usepackage{citehack}
\usepackage{amsmath,amsthm}
\usepackage{amsfonts}
\usepackage{amssymb}
\usepackage{longtable}
\usepackage[matrix,arrow,curve]{xy}

\usepackage{lipsum}

\def\W{\mathcal{W}}
\def\K{\mathcal{K}}

\usepackage{stackrel}

\newcommand{\be}{\begin{equation}}
\newcommand{\ee}{\end{equation}}

\newcommand{\nn}{\nonumber \\}

 \newcommand{\res}{\mbox{\rm Res}}

\renewcommand{\hom}{\mbox{\rm Hom}}

\newcommand{\wt}{\mbox{\rm wt}\ }

\newcommand{\one}{\mathbf{1}}

\newcommand{\nc}{\newcommand}
\nc{\cali}{\mathcal}
\nc{\on}{\operatorname}
\nc{\Wick}{{\mb :}}

\nc{\ddz}{\frac{\partial}{\partial z}}
\nc{\ch}{\mbox{ch}}
\nc{\Oo}{{\cali O}}
\nc{\cond}{|\,}
\nc{\bib}{\bibitem}
\nc{\pone}{\Pro^1}
\nc{\pa}{\partial}
\nc{\arr}{\rightarrow}
\nc{\larr}{\longrightarrow}
\nc{\ket}{\rangle}
\nc{\bra}{\langle}
\nc{\gam}{\bar{\gamma}}
\nc{\q}{\widetilde{Q}}
\nc{\ep}{\epsilon}
\nc{\su}{\widehat{{\mf s}{\mf l}}_2}
\nc{\sw}{{\mf s}{\mf l}}
\nc{\h}{{\mf h}}
\nc{\n}{{\mf n}}
\nc{\ab}{\mf{a}}
\nc{\is}{{\mb i}}
\nc{\js}{{\mb j}}
\nc{\bi}{\bibitem}
\nc{\He}{{\cali H}}
\nc{\inv}{^{-1}}
\nc{\ol}{\overline}
\nc{\wh}{\widehat}
\nc{\dst}{\displaystyle}

\nc{\delt}{\partial_t}
\nc{\ddt}{\frac{\partial}{\partial t}}
\nc{\delx}{\partial_x}
\nc{\mb}{\mathbf}
\nc{\mf}{\mathfrak}

\nc{\mbb}{\mathbb}
\nc{\Ctt}{\C((t))}
\nc{\Ct}{\C[t,t\inv]}

\nc{\ghat}{\wh{\g}}

\nc{\un}{\underline}
\nc{\mc}{\mathcal}
\nc{\BB}{{\mc B}}
\nc{\bb}{{\mf b}}
\nc{\kk}{{\mf k}}
\nc{\frob}{\times}
\nc{\sm}{\setminus}
\nc{\Pp}{{\mathbb P}^1}
\nc{\Aa}{{\mc A}}

\nc{\AutO}{\on{Aut}\Oo}
\nc{\AUTO}{\un{\on{Aut}}\Oo}
\nc{\AUTK}{\un{\on{Aut}}\K}
\nc{\Heout}{\He_{\out}}
\nc{\Hetil}{{\widetilde\He}}
\nc{\wb}{\overline}

\nc{\Res}{\on{Res}}
\nc{\pitil}{\Pi}
\nc{\Ctil}{\wt{C}}
\nc{\auto}{\on{Aut} \Oo}
\nc{\phitil}{\wt{\phi}}
\nc{\gz}{\g_{\vec z}}
\nc{\tensorM}{\bigotimes_{i=1}^N{\mathbb M}_i}
\nc{\tensorW}{\bigotimes_{i=1}^N W_{\nu_i,k}}
\nc{\out}{\on{out}}

\nc{\m}{{\mathfrak m}}

\nc{\gx}{\g^0_{\vec x}}

\nc{\hx}{\He^0_{\vec x}}
\nc{\tensorpi}{\pi_{\nu_1,\ldots,\nu_N}^\kappa}
\nc{\Phizw}{\Phi_{\vec w}({\vec z})}
\nc{\Pro}{{\mathbb P}}

\nc{\De}{\Delta}

\nc{\us}{\underset}
\nc{\Ll}{\mc L}
\nc{\dR}{\on{dR}}

\nc{\T}{{\mc T}}

\nc{\Xn}{\overset{\circ}X{}^n} \nc{\Dn}{\overset{\circ}D{}^n}
\nc{\Dxn}{\overset{\circ}D{}^n_x} \nc{\varphitil}{\wt{\varphi}}

\nc{\lf}{{\mf l}}
\nc{\GL}{{}^L G}
\nc{\Vir}{\on{Vir}}

\begin{document}

\title[The simplest cohomological invariants for vertex algebras]  
{The simplest cohomological invariants for vertex algebras} 
%%
 % please provide
                                % an abbreviated title 
%%\newcommand{\half}{\frac{1}{2}}
%%%%%%%%%%%%%%%%%%%%%%%%%%%%%%%%%%%%%%%%%%%%%%%%%%%%%%%%%%%%%%%%%%%%%%%%%%%%%%%
\author{A. Zuevsky} 
\address{Institute of Mathematics \\ Czech Academy of Sciences\\ \\ Zitna 25, 11567 \\ Prague\\ Czech Republic}

\email{zuevsky@yahoo.com}

%%%%%%%%%%%%%%%%%%%%%%%%%%%%%%%%%%%%%%%%%%%%%%%%%%%%%%%%%%%%%%%%%%%%%%%%%
% You may repeat \author \address as often as necessary                 %
%%%%%%%%%%%%%%%%%%%%%%%%%%%%%%%%%%%%%%%%%%%%%%%%%%%%%%%%%%%%%%%%%%%%%%%%%
\begin{abstract}
For the double complex structure of grading-restricted vertex algebra cohomology defined in \cite{Huang}, 
we introduce a multiplication of elements of double complex spaces. %, and prove its properties. 
We show that the orthogonality and bi-grading conditions applied on double complex spaces, provide in relation 
among mappings and actions of co-boundary operators. 
Thus, we endow the double complex spaces with structure of bi-graded differential algebra. 
We then introduce the simples cohomology classes for a grading-restricted vertex algebra, and 
 show their independence on the choice of mappings from double complex spaces. 
We prove that its cohomology class does not depend on mappings representing of 
the double complex spaces.  
Finally, we show that the orthogonality relations together with the bi-grading condition 
bring about generators and commutation relations for a continual Lie algebra. % \cite{saver}. 
AMS Classification: 53C12, 57R20, 17B69 
\end{abstract}

\keywords{Vertex algebras, cohomological invariants, cohomology classes}
\vskip12pt  % insert '\vskip12pt' while using '\twocolumn' command
%\vskip28pt % if there is no keywords

\maketitle
%%%%%%%%%%%%%%%%%%%%%%%%%%%%%%%%%%%%%%%%%%%%%%%%%%%%%%%%%%%%%%%%%%%%%%%%%%%%%%%%%%%%%%%%%%%%%%%%%%%%%%%%%%%%%%
\section{Introduction: $\overline{W}$-valued rational functions}
%%%%%
\label{valued}
In \cite{Huang} the cohomology theory for a grading-restricted vertex algebra \cite{K} (see  Appendix \ref{grading})
was introduced. 
%% (see also \cite{FQ}).  
%%
The definition of double complex spaces and co-boundary operators, 
uses an interpretation of vertex algebras in terms of rational functions constructed from    
 matrix elements \cite{H2} for a grading-restricted vertex algebra. 
%%%%
The notion of composability (see Section \ref{composable}) of double complex space elements with a number of vertex operators, is essentially 
involved in the formulation. 
Then the cohomology of such complexes defines in the standard way a cohomology of a grading-restricted vertex 
algebras. 
%%%%
%%
It is an important problem to study possible cohomological classes for vertex algebras. 
In this paper we do the first steps to discover simplest cohomological
 invariants associated to the setup described above. 
For that purpose we first endow the double complex spaces with natural product, derive a counterpart of Leibniz formula 
for the action of co-boundary operators. 
Then we introduce the notion of a cohomological class for a vertex algebra. 
The orthogonality condition of double complex space is then defined. 
We show that the orthogonality being applied to the double complex spaces leads to relations among mappings and actions 
of co-boundary operators.  
The simplest non-vanishing cohomology classes for a grading-restricted vertex algebra is then derived. 
We show that such classes are independent of the choice of elements of the double complex spaces. 
Finally, we discuss occurring relations of a vertex algebra double complex relations with
 a continual Lie algebra \cite{saver}. 
For further applications of material introduced  in this paper, we would mention 
 the natural question of searching for more general cohomological invariants for a grading-restricted vertex algebra. 
Concerning possible applications, one can use the cohomological classes we derive to compute 
higher cohomologies of grading-restricted vertex algebras.   
%%

%%%%%%%%%%%%%%%%%%%%%%%%%%%%%%%%%%%%%%%%%%%%%%%%%%%%%%%%%%%%%%%%%%%%%%%%%%%%%%%%%%%%%%%%%%%%
Let $V$ be a grading-restricted 
vertex algebra, and $W$ a a grading-restricted generalized $V$-module (see Appendix \ref{grading}). 
 One defines the configuration space %are
 \cite{Huang}: 
\[
F_{n}\C=\{(z_{1}, \dots, z_{n})\in \C^{n}\;|\; z_{i}\ne z_{j}, i\ne j\},
\] 
for $n\in \Z_{+}$.
%
%%%%%%%%%%%%%%%%%%%%%%%%%%%%%%%%%%%%%%%%%%%%%%%%%%%%%%%%%%%%%%%%%%%%%%%%%%
%%%%%%%%%%%%%%%%%%%%%%%%%%%%%%%%%%%%%%%%%%%%%%%%%%%%%%%%%%%%%%%%%%%%%%%%%%
\begin{definition}
 A $\overline{W}$-valued rational function $\F$ in $(z_{1}, \dots, z_{n})$ 
with the only possible poles at 
$z_{i}=z_{j}$, $i\ne j$, 
is a map 
\begin{eqnarray*}
 \F: F_{n}\C &\to& \overline{W},   
\\
 (z_{1}, \dots, z_{n}) &\mapsto& \F(z_{1}, \dots, z_{n}),   
\end{eqnarray*} 
such that for any $w'\in W'$,  
\begin{equation}
\label{def}
\langle w', \F(z_{1}, \dots, z_{n}) \rangle,
\end{equation}
is a rational function in $(z_{1}, \dots, z_{n})$  
with the only possible poles  at 
$z_{i}=z_{j}$, $i\ne j$. 
Such map is called in what fallows $\overline{W}$-valued rational function in
$(z_{1}, \dots, z_{n})$ with possible other poles. 
Denote the space of all $\overline{W}$-valued rational functions in
$(z_{1}, \dots, z_{n})$ by $\overline{W}_{z_{1}, \dots, z_{n}}$. 
\end{definition}
%%%%%%%%%%%%%%%%%%%%%%%%%%%%%%%%%%%%%%%%%%%%%%%%%%%%%%%%%%%%%%%%%%%%%
%%
 %Here $R(.)$ denotes the following (cf. \cite{Huang}).   
%%
Namely, 
if a meromorphic function $f(z_{1}, \dots, z_{n})$ on a region in $\C^{n}$
can be analytically extended to a rational function in $(z_{1}, \dots, z_{n})$, 
then the notation 
%%
%\[
%%
$R(f(z_{1}, \dots, z_{n}))$, 
%%
%\]
%%
 is used to denote such rational function. 
Note that the set of a grading-restricted vertex algebra elements $(v_1, \ldots, v_n)$ associated with 
corresponding $(z_1, \ldots, z_n)$ play the role of non-commutative parameters for a function $\F$ in \eqref{def}. 
%%
%%%%%%%%%%%%%%%%%%%%%%%%%%%%%%%%%%%%%%%%%%%%%%%%%%%%%%%%%%%%%%%%%%%%%%%
%
%%%%%%%%%%%%%%%%%%%%%%%%%%%%%%%%%%%%%%%%%%%%%%%%%%%%%%%%%%%%%%%%%%%
%%%%%%%%%%%%%%%%%%%%%%%%%%%%%%%%%%%%%%%%%%%%%%%%%%%%%%%%%%%%%%%%%%%
%%
Let us introduce the definition of a $\W_{z_1, \ldots, z_n}$-space: 
%%
%%%%%%%%%%%%%%%%%%%%%%%%%%%%%%%%%%%%%%%%%%%%%%%%%%%%%%%%%%%%%%%%%%%%%%%
\begin{definition}
\label{wspace}
We define the space $\W_{z_1, \dots, z_n}$ of 
  $\overline{W}_{z_{1}, \dots, z_{n}}$-valued rational forms $\Phi$  
with each vertex algebra element entry $v_i$, $1 \le i \le n$
of a quasi-conformal grading-restricted vertex algebra $V$ tensored with power $\wt(v_i)$-differential of 
corresponding formal parameter $z_i$, i.e., 
\begin{eqnarray}
\label{bomba}
&& 
\Phi \left(v_{1}, z_1; \ldots;  
  v_{n}, z_n\right)
\nn
 && \qquad 
= \F \left(dz_1^{{\rm \wt}(v_1)} \otimes v_{1}, z_1; \ldots;
 dz_n^{{\rm \wt}(v_n)} \otimes  v_{n}, z_n\right) \in \W_{z_1, \dots, z_n}. 
\end{eqnarray}
where $\F \in \overline{W}_{z_1, \dots, z_n}$. 
\end{definition}
%%
%%
%%%%%%%%%%%%%%%%%%%%%%%%%%%%%%%%%%%%%%%%%%%%%%%%%%%%%%%%%%%%%%%%%%%%%%%%%%%%
%%%%%%%%%%%%%%%%%%%%%%%%%%%%%%%%%%%%%%%%%%%%%%%%%%%%%%%%%%%%%%%%%%%%%%%%%%%%
\begin{definition}
One defines an action of $S_{n}$ on the space $\hom(V^{\otimes n}, 
%\overline{
\W%}
_{z_{1}, \dots, z_{n}})$ of linear maps from 
$V^{\otimes n}$ to $%\overline{
W
%}
_{z_{1}, \dots, z_{n}}$ by 
 \begin{equation}
\label{sigmaction}
\sigma(\Phi)(v_{1}\otimes \cdots\otimes v_{n})(z_{1}, \dots, z_{n}) ,  
=\Phi(v_{\sigma(1)}\otimes \cdots\otimes v_{\sigma(n)})(z_{\sigma(1)}, \dots, z_{\sigma(n)}),
\end{equation} 
for $\sigma\in S_{n}$ and $v_{1}, \dots, v_{n}\in V$, $\Phi \in \W_{z_{1}, \dots, z_{n}}$.
We will use the notation $\sigma_{i_{1}, \dots, i_{n}}\in S_{n}$, to denote the 
the permutation given by $\sigma_{i_{1}, \dots, i_{n}}(j)=i_{j}$,  
for $j=1, \dots, n$.
\end{definition}
\begin{definition}
\label{wspace}
For $n\in \Z_{+}$,  
a linear map 
\[
\F(v_{1}, z_{1};  \ldots ; v_{n},  z_{n})
= V^{\otimes n}\to 
\W_{z_{1}, \dots, z_{n}}, 
\]
 is said to have
the  $L_V(-1)$-derivative property if
%%%
\begin{equation}
\label{lder1}
(i) \qquad %\langle w', 
\partial_{z_{i}} \F (v_{1}, z_{1};  \ldots ; v_{n},  z_{n})%\rangle
= 
%%
%\langle w',  
%%
\F(v_{1}, z_{1};  \ldots; L_{V}(-1)v_{i}, z_i; \ldots ; v_{n},  z_{n}), %\rangle, 
\end{equation}
for $i=1, \dots, n$, $(v_{1}, \dots, v_{n}) \in V$, $w'\in W$, 
and  
\begin{eqnarray}
\label{lder2}
(ii) \qquad \sum\limits_{i=1}^n\partial_{z_{i}} %\langle w', 
\F(v_{1}, z_{1};  \ldots ; v_{n},  z_{n})= %\rangle=
%%
%\langle w',
%%
 L_{W}(-1).\F(v_{1}, z_{1};  \ldots ; v_{n},  z_{n}), % \rangle, 
\end{eqnarray}
with some action $"."$ of $L_{W}(-1)$  on $\F(v_{1}, z_{1};  \ldots ; v_{n},  z_{n})$. %and 
%and  $(v_{1}, \dots, v_{n})\in V$. 
%%
\end{definition}
\begin{definition}
A linear map 
\[
\F: V^{\otimes n} \to \W_{z_{1}, \dots, z_{n}} 
\]
 has the  $L_W{(0)}$-conjugation property if for $(v_{1}, \dots, v_{n}) \in V$,
%%
%$w'\in W'$, 
%%
$(z_{1}, \dots, z_{n})\in F_{n}\C$, and $z\in \C^{\times}$, such that 
$(zz_{1}, \dots, zz_{n})\in F_{n}\C$,
\begin{eqnarray}
\label{loconj}
%%
%%\langle w',
%%
 z^{L_{W}(0)}   
\F \left(v_{1}, z_1; \ldots; v_{n}, z_{n} \right) %\rangle 
%%
%%\nn
%%
%%&&
%%
=%\langle w', 
 \F\left(z^{ L_V{(0)} } v_{1}, zz_{1};  \ldots ;  z^{L_V{(0)} } v_{n},  zz_{n}\right). %\rangle.
\end{eqnarray} 
\end{definition}
%%

%%%%%%%%%%%%%%%%%%%%%%%%%%%%%%%%%%%%%%%%%%%%%%%%%%%%%%%%%%%%%%%%%%%%%%%%%%%%%%%%%
%%%%%%%%%%%%%%%%%%%%%%%%%%%%%%%%%%%%%%%%%%%%%%%%%%%%%%%%%%%%%%%%%%%%%%%%%%%%%%%%%%%
%%
\subsection{E-elements}
 For $w\in W$, the $\overline{W}$-valued function
$E^{(n)}_{W}(v_{1}\otimes \cdots\otimes v_{n}; w)$ is 
given by 
$$
E^{(n)}_{W}(v_{1}\otimes \cdots\otimes v_{n}; w)(z_{1}, \dots, z_{n})
=E(Y_{W}(v_{1}, z_{1})\cdots Y_{W}(v_{n}, z_{n})w),  
$$
where an element $E(.)\in \overline{W}$ is 
given by 
\[
\langle w',E(.)\rangle =R(\langle w', . \rangle),
\]
 and $R(.)$ denotes the rationalization in the sense of \cite{Huang}.  
Namely, 
if a meromorphic function $f(z_{1}, \dots, z_{n})$ on a region in $\C^{n}$
can be analytically extended to a rational function in $(z_{1}, \dots, z_{n})$, 
then the notation $R(f(z_{1}, \dots, z_{n}))$ is used to denote such rational function. 
One defines  
\[
E^{W; (n)}_{WV}(w; v_{1}\otimes \cdots\otimes v_{n})
=E^{(n)}_{W}(v_{1}\otimes \cdots\otimes v_{n}; w),
\]
where 
$E^{W; (n)}_{WV}(w; v_{1}\otimes \cdots\otimes v_{n})$ is 
an element of $\overline{W}_{z_{1}, \dots, z_{n}}$.
One defines
\[
\Phi\circ \left(E^{(l_{1})}_{V;\;\one}\otimes \cdots \otimes E^{(l_{n})}_{V;\;\one}\right): 
V^{\otimes m+n}\to 
\overline{W}_{z_{1},  \dots, z_{m+n}},
\] 
by
\begin{eqnarray*}
\lefteqn{(\Phi\circ (E^{(l_{1})}_{V;\;\one}\otimes \cdots \otimes 
E^{(l_{n})}_{V;\;\one}))(v_{1}\otimes \cdots \otimes v_{m+n-1})}\nn
&&=E(\Phi(E^{(l_{1})}_{V; \one}(v_{1}\otimes \cdots \otimes v_{l_{1}})\otimes \cdots
%%\nn 
%%&&\quad\quad\quad\quad\quad \otimes 
%%
E^{(l_{n})}_{V; \one}
(v_{l_{1}+\cdots +l_{n-1}+1}\otimes \cdots 
\otimes v_{l_{1}+\cdots +l_{n-1}+l_{n}}))),  
\end{eqnarray*}
and 
\[
E^{(m)}_{W}\circ_{m+1}\Phi: V^{\otimes m+n}\to 
\overline{W}_{z_{1}, \dots, 
z_{m+n-1}},
\]
 is given by 
\begin{eqnarray*}
\lefteqn{
(E^{(m)}_{W}\circ_{m+1}\Phi)(v_{1}\otimes \cdots \otimes v_{m+n})
}\nn
&&
=E(E^{(m)}_{W}(v_{1}\otimes \cdots\otimes v_{m};
\Phi(v_{m+1}\otimes \cdots\otimes v_{m+n}))). 
\end{eqnarray*}
Finally,  
\[
E^{W; (m)}_{WV}\circ_{0}\Phi: V^{\otimes m+n}\to 
\overline{W}_{z_{1}, \dots, 
z_{m+n-1}},
\]
 is defined by 
\begin{eqnarray*}
(E^{W; (m)}_{WV}\circ_{0}\Phi)(v_{1}\otimes \cdots \otimes v_{m+n})
%%\nn
%%&&
 =E(E^{W; (m)}_{WV}(\Phi(v_{1}\otimes \cdots\otimes v_{n})
; v_{n+1}\otimes \cdots\otimes v_{n+m})). 
\end{eqnarray*}
In the case that $l_{1}=\cdots=l_{i-1}=l_{i+1}=1$ and $l_{i}=m-n-1$, for some $1 \le i \le n$,
we will use $\Phi\circ_{i} E^{(l_{i})}_{V;\;\one}$ to 
denote $\Phi\circ (E^{(l_{1})}_{V;\;\one}\otimes \cdots 
\otimes E^{(l_{n})}_{V;\;\one})$.

%%%%%%%%%%%%%%%%%%%%%%%%%%%%%%%%%%%%%%%%%%%%%%%%%%%%%%%%%%%%%%%%%%%%%%%%%%%%
%%%%%%%%%%%%%%%%%%%%%%%%%%%%%%%%%%%%%%%%%%%%%%%%%%%%%%%%%%%%%%%%%%%%%%%%%%%%
\subsection{Maps composable with vertex operators}
\label{composable}
Since $\overline{W}$-valued rational functions above are valued in $\overline{W}$,
 for $z\in \C^{\times}$, $u, v\in V$, $w\in W$, $Y_{V}(u, z)v\in \overline{V}$, 
and $Y_{W}(u, z)v\in \overline{W}$, 
one might not be able to compose in general a linear map from 
a tensor power of $V$ to $\overline{W}_{z_{1}, \dots, z_{n}}$ with 
vertex operators. 
 Thus in \cite{Huang} they consider linear maps
from tensor powers of $V$ to $\overline{W}_{z_{1}, \dots, z_{n}}$ 
such that these maps can be composed with vertex operators in the sense mentioned above.  
\begin{definition}
\label{mapco}
For a $V$-module $W=\coprod_{n\in \C}W_{(n)}$ and $m\in \C$, 
let $P_{m}: \overline{W}\to W_{(m)}$ be 
the projection from $\overline{W}$ to $W_{(m)}$. 
%%
%\begin{defn}\label{composable}
%%
Let $\Phi: V^{\otimes n}\to 
\overline{W}_{z_{1}, \dots, z_{n}}$ be a linear map. For $m\in \N$, 
$\Phi$ is said \cite{Huang} to be composable with $m$ vertex operators if 
the following conditions are satisfied:
\begin{enumerate}
\item Let $l_{1}, \dots, l_{n}\in \Z_+$ such that $l_{1}+\cdots +l_{n}=m+n$,
$v_{1}, \dots, v_{m+n}\in V$ and $w'\in W'$. Set 
 \begin{eqnarray*}\label{psi-i}
\Psi_{i}
&
=
&
E^{(l_{i})}_{V}(v_{k_1} %{l_{1}+\cdots +l_{i-1}+1}
\otimes
\cdots \otimes v_{k_i}%{l_{1}+\cdots +l_{i-1}+l_{i}}
; \one_{V})   
%%\nn
%% &&\quad\quad\quad\quad\quad\quad\quad\quad
%%
(z_{k_1},%{l_{1}+\cdots +l_{i-1}+1}-\zeta_{i},
 \dots, 
z_{k_i}%{l_{1}+\cdots +l_{i-1}+l_{i}}-\zeta_{i}
),
%%\nn
\end{eqnarray*}
where ${k_1}={l_{1}+\cdots +l_{i-1}+1}$, ..., $v_{k_i}={l_{1}+\cdots +l_{i-1}+l_{i}}$, 
for $i=1, \dots, n$. Then there exist positive integers $N^n_m(v_{i}, v_{j})$
depending only on $v_{i}$ and $v_{j}$ for $i, j=1, \dots, k$, $i\ne j$ such that the series 
\[
\sum_{r_{1}, \dots, r_{n}\in \Z}\langle w', 
(\Phi(P_{r_{1}}\Psi_{1}\otimes \dots\otimes
P_{r_n} \Psi_{n}))(\zeta_{1}, \dots, 
\zeta_{n})\rangle,
\] 
is absolutely convergent  when 
$|z_{l_{1}+\cdots +l_{i-1}+p}-\zeta_{i}| 
+ |z_{l_{1}+\cdots +l_{j-1}+q}-\zeta_{i}|< |\zeta_{i}
-\zeta_{j}|$, 
for $i,j=1, \dots, k$, $i\ne j$ and for $p=1, 
\dots,  l_i$ and $q=1, \dots, l_j$. 
The sum must be analytically extended to a
rational function
in $(z_{1}, \dots, z_{m+n})$,
 independent of $(\zeta_{1}, \dots, \zeta_{n})$, 
with the only possible poles at 
$z_{i}=z_{j}$, of order less than or equal to 
$N^n_m(v_{i}, v_{j})$, for $i,j=1, \dots, k$,  $i\ne j$. 

\item For $v_{1}, \dots, v_{m+n}\in V$, there exist 
positive integers $N^n_m(v_{i}, v_{j})$, depending only on $v_{i}$ and 
$v_{j}$, for $i, j=1, \dots, k$, $i\ne j$, such that for $w'\in W'$, and 
${\bf v}_{n,m}=(v_{1+m}\otimes \cdots\otimes v_{n+m})$,  ${\bf z}_{n,m}=(z_{1+m}, \dots, z_{n+m})$, such that 
%%
%%\begin{eqnarray*}
%%
\[
\sum_{q\in \C}\langle w', 
(E^{(m)}_{W}(v_{1}\otimes \cdots\otimes v_{m}; 
%%\nn
%%&&\quad\quad\quad 
%%
P_{q}((\Phi({\bf v}_{n,m}))({\bf z}_{n,m})))\rangle, 
%\end{eqnarray*}
\]
is absolutely convergent when $z_{i}\ne z_{j}$, $i\ne j$
$|z_{i}|>|z_{k}|>0$ for $i=1, \dots, m$, and 
$k=m+1, \dots, m+n$, and the sum can be analytically extended to a
rational function 
in $(z_{1}, \dots, z_{m+n})$ with the only possible poles at 
$z_{i}=z_{j}$, of orders less than or equal to 
$N^n_m(v_{i}, v_{j})$, for $i, j=1, \dots, k$, $i\ne j$,. 
\end{enumerate}
%%
%%%%%%%%%%%%%%%%%%%%%%%%%%%%%%%%%%%%%%%%%%%%%%%%%%%%%%%%%%%%%%%%%%%%%%%%%%%%%%%%
\end{definition}
In \cite{Huang} one finds:
%%
%%%%%%%%%%%%%%%%%%%%%%%%%%%%%%%%%%%%%%%%%%%%%%%%%%%%%%%%%%%%%%%%%%%%%%%%%%%%%%%%%%
\begin{proposition}  %{\rm \cite{Huang}}  
\label{}
The subspace of $\hom(V^{\otimes n}, 
%\overline{
\W
%}
_{z_{1}, \dots, z_{n}})$ consisting of linear maps
having
the $L(-1)$-derivative property, having the $L(0)$-conjugation property
or being composable with $m$ vertex operators is invariant under the 
action of $S_{n}$.
\end{proposition}
%%
%%%%%%%%%%%%%%%%%%%%%%%%%%%%%%%%%%%%%%%%%%%%%%%%%%%%%%%%%%%%%%%%%%%%%%%%%%%%%%%%%%%%%%%%%%%%%%%%%%%
%%%%%%%%%%%%%%%%%%%%%%%%%%%%%%%%%%%%%%%%%%%%%%%%%%%%%%%%%%%%%%%%%%%%%%%%%%%%%%%%%%%%%%%%%%%%%%%%%%%
\section{Chain complexes and cohomologies}
\label{complexes}
%%%%%%%%%%%%%%%%%%%%%%%%%%%%%%%%%%%%%%%%%%%%%%%%%%%%%%%%%%%%%%%%%%%%%%%%%%%%%%%%%%%%%%%%%%%%%%%%%%%
%%%%%%%%%%%%%%%%%%%%%%%%%%%%%%%%%%%%%%%%%%%%%%%%%%%%%%%%%%%%%%%%%%%%%%%%%%%%%%%%%%%%%%%%%%%%%%%%%%%
Let us recall the definition of shuffles \cite{Huang}. 
\begin{definition}
For $l \in \N$ and $1\le s \le l-1$, let $J_{l; s}$ be the set of elements of 
$S_{l}$ which preserve the order of the first $s$ numbers and the order of the last 
$l-s$ numbers, i.e., 
\[
J_{l, s}=\{\sigma\in S_{l}\;|\;\sigma(1)<\cdots <\sigma(s),\;
\sigma(s+1)<\cdots <\sigma(l)\}.
\]
The elements of $J_{l; s}$ are called shuffles. Let $J_{l; s}^{-1}=\{\sigma\;|\;
\sigma\in J_{l; s}\}$.
\end{definition}
Now we introduce the notion of a $C^n_m(V, \W)$-space: 
%%%%%%%%%%%%%%%%%%%%%%%%%%%%%%%%%%%%%%%%%%%%%%%%%%%%%%%%%%%%%%%%%%%
%% 
\begin{definition}
\label{kakashka}
Let $V$ be a vertex operator algebra and $W$ a $V$-module. 
For $n\in \Z_{+}$, let $ {C}_{0}^{n}(V, \W)$ be the vector space of all 
linear maps from $V^{\otimes n}$ to $\W_{z_{1}, \dots, z_{n}}$  
satisfying the $L(-1)$-derivative property and the $L(0)$-conjugation property. 
For $m$, $n\in \Z_{+}$, 
let $ {C}_{m}^{n}(V, \W)$ be the vector spaces of all 
linear maps from $V^{\otimes n}$ to $\W_{z_{1}, \dots, z_{n}}$ 
composable with $m$ vertex operators, and satisfying the $L(-1)$-derivative
property, the $L(0)$-conjugation property, and such that 
\begin{equation}
\label{shushu}
\sum_{\sigma\in J_{l; s}^{-1}}(-1)^{|\sigma|}
\sigma\left(\Phi(v_{\sigma(1)}\otimes \cdots \otimes v_{\sigma(l)})\right)=0.
\end{equation}
\end{definition}
%%
%%%%%%%%%%%%%%%%%%%%%%%%%%%%%%%%%%%%%%%%%%%%%%%%%%%%%%%%%%%%%%%%%%%%%%%%%%%%%%%%%%%%%%%
%\begin{remark}
%%
 %In \cite{Zu},
%% 
 Using a generalization of the construciton of the vertex algebra bundle and coordinate-free formulation of 
vertex operators in \cite{BZF} for the case of $\W$-valued forms, we obtain 
following 
\begin{lemma}
\label{popa}
 that an element \eqref{bomba} of $C^n_m(V, \W)$ is invariant with respect the group
 ${\rm Aut}_{z_1, \ldots, z_n}\Oo^{(n)}$ of $n$-dimensional independent changes of formal parameters
\[
(z_1, \ldots, z_n) \mapsto (\rho_1(z_1, \ldots, z_n), \ldots, \rho_n(z_1, \ldots, z_n)). 
\]
\hfill $\square$
\end{lemma}
%%
%%
%%%%%%%%%%%%%%%%%%%%%%%%%%%%%%%%%%%%%%%%%%%%%%%%%%%%%%%%%%%%%%%%%%%%%%%%%%%%%%%
We also find in \cite{Huang}
\begin{proposition} 
Let $ {C}_{m}^{0}(V, \W)=\W$. Then we have 
 $C_{m}^{n}(V, \W)\subset  {C}_{m-1}^{n}(V, \W)$, 
for $m\in \Z_{+}$. 
\end{proposition}
%%
%%%%%%%%%%%%%%%%%%%%%%%%%%%%%%%%%%%%%%%%%%%%%%%%%%%%%%%%%%%%%%%%%%%%%%%%%%%%%%%%%%%%%%%%5
In \cite{Huang} the co-boundary operator for the double complex spaces ${C}_{m}^{n}(V, \W)$ was introduced: 
\begin{equation}
\label{hatdelta}
 {\delta}^{n}_{m}:  {C}_{m}^{n}(V, \W)
\to  {C}_{m-1}^{n+1}(V, \W). 
\end{equation} 
For $\Phi \in {C}_{m}^{n}(V, \W)$, it is given by 
\begin{equation}
\label{marsha}
{\delta}^{n}_{m}(\Phi)
=E^{(1)}_{W}\circ_{2} \Phi
+\sum_{i=1}^{n}(-1)^{i}\Phi\circ_{i} E^{(2)}_{V; \one}
 +(-1)^{n+1}
\sigma_{n+1, 1, \dots, n}(E^{(1)}_{W}\circ_{2}
\Phi),   %%\nn
\end{equation}
where $\circ_i$ is defined in Subsection \ref{valued}.    
%%
%%%%%%%%%%%%%%%%%%%%%%%%%%%%%%%%%%%%%%%%%%%%%%%%%%%%%%%%%%%%%%%%%%%%%%%%%%%%%5
Explicitly, 
for $v_{1}, \dots, v_{n+1}\in V$, $w'\in W'$ 
and $(z_{1}, \dots, z_{n+1})\in 
F_{n+1}\C$, 
\begin{eqnarray*}
\lefteqn{\langle w', (( {\delta}^{n}_{m}(\Phi))(v_{1}\otimes \cdots\otimes v_{n+1}))
(z_{1}, \dots, z_{n+1})\rangle}\nn
&&=R(\langle w', Y_{W}(v_{1}, z_{1})(\Phi(v_{2}\otimes \cdots\otimes v_{n+1}))
(z_{2}, \dots, z_{n+1})\rangle)\nn
&&\quad +\sum_{i=1}^{n}(-1)^{i}R(\langle w', 
(\Phi(v_{1}\otimes \cdots \otimes v_{i-1} \otimes 
Y_{V}(v_{i}, z_{i}-z_{i+1})v_{i+1}\nn
&&\quad\quad\quad\quad\quad\quad\quad\quad\quad 
\quad\quad\otimes \cdots \otimes v_{n+1}))
(z_{1}, \dots, z_{i-1}, z_{i+1}, \dots, z_{n+1})\rangle)\nn
&&\quad + (-1)^{n+1}R(\langle w', Y_{W}(v_{n+1}, z_{n+1})
(\Phi(v_{1}\otimes \cdots \otimes v_{n}))(z_{1}, \dots, z_{n})\rangle).
\end{eqnarray*}
 In the case $n=2$, there is a subspace 
of $ {C}_{0}^{2}(V, \W)$ 
containing $ {C}_{m}^{2}(V, \W)$ for all $m\in \Z_{+}$ such that 
$ {\delta}^{2}_{m}$ is still defined on this subspace. 
Let $ {C}_{\frac{1}{2}}^{2}(V, \W)$ be the subspace of $ {C}_{0}^{2}(V, \W)$ 
consisting of elements $\Phi$ such that for $v_{1}, v_{2}, v_{3}\in V$, $w'\in W'$, 
\begin{eqnarray*}
&& \sum_{r\in \C}\big(\langle w', E^{(1)}_{W}(v_{1};
P_{r}((\Phi(v_{2}\otimes v_{3}))(z_{2}-\zeta, z_{3}-\zeta)))(z_{1}, \zeta)\rangle\nn
&&\quad+\langle w', (\Phi(v_{1}\otimes P_{r}((E^{(2)}_{V}(v_{2}\otimes v_{3}; \one))
(z_{2}-\zeta, z_{3}-\zeta))))
(z_{1}, \zeta)\rangle\big), 
\end{eqnarray*}
and 
\begin{eqnarray*}
\lefteqn{\sum_{r\in \C}\big(\langle w', 
(\Phi(P_{r}((E^{(2)}_{V}(v_{1}\otimes v_{2}; \one))(z_{1}-\zeta, z_{2}-\zeta))
\otimes v_{3}))
(\zeta, z_{3})\rangle}\nn
&&\quad +\langle w', 
E^{W; (1)}_{WV}(P_{r}((\Phi(v_{1}\otimes v_{2}))(z_{1}-\zeta, z_{2}-\zeta));
v_{3}))(\zeta, z_{3})\rangle\big)
\end{eqnarray*}
are absolutely convergent in the regions $|z_{1}-\zeta|>|z_{2}-\zeta|, |z_{2}-\zeta|>0$ and 
$|\zeta-z_{3}|>|z_{1}-\zeta|, |z_{2}-\zeta|>0$, respectively, 
and can be analytically extended to 
rational functions in $z_{1}$ and $z_{2}$ with the only possible poles at
$z_{1}, z_{2}=0$ and $z_{1}=z_{2}$.
It is clear that 
${C}_{m}^{2}(V, \W)\subset {C}_{\frac{1}{2}}^{2}(V, \W)$
for $m\in \Z_{+}$. 
 The co-boundary operator 
\begin{equation}
\label{halfdelta}
{\delta}^{2}_{\frac{1}{2}}: {C}_{\frac{1}{2}}^{2}(V, \W)
\to {C}_{0}^{3}(V, \W),
\end{equation}
is defined in \cite{Huang} by 
\begin{eqnarray}
\label{halfdelta1}
&& {\delta}^{2}_{\frac{1}{2}}(\Phi) 
%(v_{1}\otimes v_{2} \otimes v_{3}))(z_{1}, z_{2}, z_{3})\rangle
%%\nn
 %%&&
= E^{(1)}_{W} \circ_2 \Phi %(v_{1}; \Phi)  
%%&&
%%\quad \quad 
%%
+ \sum\limits_{i=1}^2 (-1)^i  E^{(2)}_{V, \one_V}  \circ_i \Phi
%%
%%+ E^{(2)}_{V, \one_V}  \circ_2 \Phi 
%%+ \Phi(v_{1}\otimes E^{(2)}_{V}(v_{2}\otimes v_{3}; \one)
%%\nn
%% &&\quad 
%%- E^{(2)}_{V, \one_V} \circ_1 \Phi %(E^{(2)}_{V}(v_{1}\otimes v_{2}; \one)) \otimes v_{3}
%%\nn
%% &&\quad \quad 
+ E^{W; (1)}_{WV}\circ_2\Phi, 
%%
%%\end{eqnarray}
%%
%%or, explicitly, 
%%
%%\begin{eqnarray*}
%%
%%\label{halfdeltadef}
\nn
&&\langle w', (({\delta}^{2}_{\frac{1}{2}}(\Phi))
(v_{1}\otimes v_{2} \otimes v_{3}))(z_{1}, z_{2}, z_{3})\rangle
\nn
 &&
=R(\langle w', (E^{(1)}_{W}(v_{1};
\Phi(v_{2}\otimes v_{3}))(z_{1}, z_{2}, z_{3})\rangle\nn
&&
\quad \quad 
+\langle w', (\Phi(v_{1}\otimes E^{(2)}_{V}(v_{2}\otimes v_{3}; \one)))
(z_{1}, z_{2}, z_{3})\rangle)
\nn
 &&\quad 
-R(\langle w', 
(\Phi(E^{(2)}_{V}(v_{1}\otimes v_{2}; \one))
\otimes v_{3}))(z_{1}, z_{2}, z_{3})\rangle
\nn
 &&\quad \quad 
+\langle w', 
(E^{W; (1)}_{WV}(\Phi(v_{1}\otimes v_{2}); v_{3}))
(z_{1}, z_{2}, z_{3})\rangle)
\end{eqnarray}
for $w'\in W'$,
$\Phi\in {C}_{\frac{1}{2}}^{2}(V, \W)$,
$v_{1}, v_{2}, v_{3}\in V$ and $(z_{1}, z_{2}, z_{3})\in F_{3}\C$. 

%%%%%%%%%%%%%%%%%%%%%%%%%%%%%%%%%%%%%%%%%%%%%%%%%%%%%%%%%%%%%%%%%%%%%%%%%%%%%%%%%%%%%%%%%%%%%%%%%%%%%%%%%%%%%
Consider the short sequence of the double complex spaces 
\begin{equation}
\label{shortseq}
0\longrightarrow C_{3}^{0}(V, \W)
\stackrel{\delta_{3}^{0}}{\longrightarrow}
C_{2}^{1}(V, \W)
\stackrel{\delta_{2}^{1}}{\longrightarrow}C_{\frac{1}{2}}^{2}(V, \W)
\stackrel{\delta_{\frac{1}{2}}^{2}}{\longrightarrow}
C_{0}^{3}(V, \W)\longrightarrow 0, 
\end{equation}
of (\ref{hatdelta}).  
%%%%
The first and last arrows are trivial embeddings and projections. 

%%
%%%%%%%%%%%%%%%%%%%%%%%%%%%%%%%%%%%%%%%%%%%%%%%%%%%%%%%%%%%%%%%%%%
In \cite{Huang} we find: 
%%%%%%%%%%%%%%%%%%%%%%%%%%%%%%%%%%%%%%%%%%%%%%%%%%%%%%%%%%%%%%%%%%%
%%
\begin{proposition}
\label{delta-square} %{\rm \cite{Huang}}
For $n\in \N$ and $m\in \Z_{+}+1$, the co-boundary operators \eqref{marsha} and \eqref{halfdelta1} satisfy 
the chain complex conditions, i.e., 
\[
 {\delta}^{n+1}_{m-1}\circ {\delta}^{n}_{m}=0,
\] 
%%
%%and 
%%
\[
{\delta}^{2}_{\frac{1}{2}}\circ {\delta}^{1}_{2}=0.
\]
\end{proposition}
Since 
\[
{\delta}_{2}^{1}( {C}_{2}^{1}(V, \W))\subset 
 {C}_{1}^{2}(V, \W)\subset 
 {C}_{\frac{1}{2}}^{2}(V, \W),
\]
the second formula follows from the first one, and 
\[
{\delta}^{2}_{\frac{1}{2}}\circ  {\delta}^{1}_{2}
= {\delta}^{2}_{1}\circ  {\delta}^{1}_{2} 
=0.
\]
%%

%%%%%%%%%%%%%%%%%%%%%%%%%%%%%%%%%%%%%%%%%%%%%%%%%%%%%%%%%%%%%%%%%%%%%5
Using the double complexes (\ref{hatdelta}) and (\ref{halfdelta}),   
 for $m\in \Z_{+}$ and $n\in \N$, one introduces in \cite{Huang} 
 the $n$-th cohomology $H^{n}_{m}(V, W)$ of a grading-restricted vertex algebra $V$ 
with coefficient in $W$, and composable with $m$ vertex operators 
to be 
\[
H_{m}^{n}(V, \W)=\ker \delta^{n}_{m}/\mbox{\rm im}\; \delta^{n-1}_{m+1}, 
\]
%%
 %%and 
%%
\[
H^{2}_{\frac{1}{2}}(V, \W)
=\ker \delta^{2}_{\frac{1}{2}}/\mbox{\rm im}\; \delta_{2}^{1}.
\] 
%%

%%%%%%%%%%%%%%%%%%%%%%%%%%%%%%%%%%%%%%%%%%%%%%%%%%%%%%%%%%%%%%%%%%%%%%%%%%%%%%%%%%%%%%%%%%%%%%%%%%%%%
%%%%%%%%%%%%%%%%%%%%%%%%%%%%%%%%%%%%%%%%%%%%%%%%%%%%%%%%%%%%%%%%%%%%%%%%%%%%%%%%%%%%%%%%%%%%%%%%%%%%%
%%%%%%%%%%%%%%%%%%%%%%%%%%%%%%%%%%%%%%%%%%%%%%%%%%%%%%%%%%%%%%%%%%%%%%%%%%%%%%%%%%%%%%%%%%%%%%%%%%%%%
\section{The $\epsilon$-product of $C^n_m(V, \W)$-spaces}
\label{productc}
In this section  we introduce %recall \cite{Zu} 
 definition of the $\epsilon$-product of 
%
%th this section we consider an application of the material of Section 
%%
%\ref{product} to %the case of 
%%
%the
%%
 double complex spaces $C^n_m(V, \W)$ with the image in another double complex space coherent with respect 
to the original differential \eqref{hatdelta}, and satisfying the symmetry \eqref{shushu}, 
$L_V(0)$-conjugation \eqref{loconj}, and $L_V(-1)$-derivative \eqref{lder1} properties 
%%
%
%%%%%%%%%%%%%%%%%%%%%%%%%%%%%%%%%%%%%%%%%%%%%%%%%%%%%%%%%%%%%%%%%%%%%%%%%%%%%%%%%%%%%%%%%%%%%%%%%%%%%%%%
%%%%%%%%%%%%%%%%%%%%%%%%%%%%%%%%%%%%%%%%%%%%%%%%%%%%%%%%%%%%%%%%%%%%%%%%%%%%%%%%%%%%%%%%%%%%%%%%%%%%%%%%
 and derive an analogue of Leibniz formula. 
%%
%%%%%%%%%%%%%%%%%%%%%%%%%%%%%%%%%%%%%%%%%%%%%%%%%%%%%%%%%%%%%%%%%%%%%%%%%%%%%%%%%%%%%%%%%%%%%%%%%%%%%%%

%%%%%%%%%%%%%%%%%%%%%%%%%%%%%%%%%%%%%%%%%%%%%%%%%%%%%%%%%%%%%%%%%%%%%%%%%%%%%%%%%%
%%%%%%%%%%%%%%%%%%%%%%%%%%%%%%%%%%%%%%%%%%%%%%%%%%%%%%%%%%%%%%%%%%%%%%%%%%%%%%%%%%
\subsection{Motivation and geometrical interpretation}
The structure of $C^n_m(V, \W)$-spaces 
%%
% \W_{z_1, \ldots, z_n}$-spaces %\overline{W}$ 
%%
is quite complicated and it is difficult to introduce algebraically a product 
of its elements.  
In order to define an appropriate product of two $C^n_m(V, \W)$-spaces 
 %\W_{z_1, \ldots, z_n}$-spaces 
%%
we first have to interpret  
them geometrically. 
Basically, a $C^n_m(V, \W)$-space
% \W_{z_1, \ldots, z_n}$-space 
%%
must be associated with a certain model space, the algebraic $\W$-language should be 
transferred to a geometrical one, two model spaces should be "connected" appropriately, and,
 finally, a product should be 
defined. 
%%  

%%
%% zdes nuzhno bolee podrobno napisat 
%%
For two  
%In order to define an appropriate product for
%%
 $\W_{x_1, \ldots, x_k}$- and 
$\W_{y_{1}, \ldots,  y_{n}}$-spaces we first associate formal complex parameters  
in the sets 
$(x_1, \ldots, x_k)$ and $(y_{1}, \ldots, y_n)$ 
to parameters of two auxiliary %auxilirary  
spaces. %$\mathcal X_a$,  $a=1$, $2$.
 Then we describe a geometric procedure to 
form a resulting model space %$\mathcal X$ 
by combining two original model spaces.  %$\mathcal X_a$. 
Formal parameters of 
$\W_{z_1, \ldots, z_{k+n}}$ should be then identified with  
parameters of the resulting space.  %$\mathcal X$.
%%%% 

%%%%%%%%%%%%%%%%%%%%%%%%%%%%%%%%%%%%%%%%%%%%%%%%%%%%%%%%%%%%%%%%%%%%%%%%%%%%%%%%%%%%%%%%%%%%%%%%%%%%%%%%%%%%
Note that %by %Definition \ref{} 
according to our assumption, $(x_1, \ldots, x_k) \in F_k\C$, and $(y_{1}, \ldots, y_{n}) \in F_{n}\C$. 
As it follows from the definition of the configuration space $F_n\C$ in Subsection \ref{valued},
 in the case of coincidence of two 
formal parameters they are excluded from $F_n\C$. 
In general, it may happen that some number $r$ of formal parameters of $\W_{x_1, \ldots, x_k}$ coincide with some 
$r$ formal parameters of $\W_{y_{1}, \ldots, y_{n}}$ on the whole $\C$ (or on a domain of definition).  
Then,   
we exclude one formal parameter from each coinciding pair. 
%%
 %and 
%%
We require that the set of formal parameters
\begin{equation}
\label{zsto}
(z_1, \ldots, z_{k+n-r})= ( \ldots, {x}_{i_1}, \ldots, {x}_{i_r}, \ldots ; 
 \ldots, \widehat {y}_{j_1},  \ldots, \widehat{y}_{j_r}, \ldots ),   
\end{equation} 
where $\; \widehat{.} \; $ denotes the exclusion of corresponding formal parameter for 
 $x_{i_l}=y_{j_l}$, $1 \le l \le r$, 
for the resulting model space %$\mathcal X$ 
 would belong to $F_{k+n-r}\C$. %, where 
We denote this operation of formal parameters exclusion by 
$\widehat{R}\;\F(x_1, \ldots, x_k; y_{1}, \ldots, y_{n}; \epsilon)$. 
%%
%%
%%
%

%%
%%%%%%%%%%%%%%%%%%%%%%%%%%%%%%%%%%%%%%%%%%%%%%%%%%%%%%%%%%%%%%%%%%%%%%%%%%%%%%%%%%%%%%%%%%%%%%%%%%%%
Now we formulate the definition of the $\epsilon$-product of two $C^n_m(V, \W)$-spaces: 
%%
%%%%%%%%%%%%%%%%%%%%%%%%%%%%%%%%%%%%%%%%%%%%%%%%%%%%%%%%%%%%%%%%%%%%%%%%%%%%%%%%%%%%%%%%%%%%%%%%%%%%%%
%%%%%%%%%%%%%%%%%%%%%%%%%%%%%%%%%%%%%%%%%%%%%%%%%%%%%%%%%%%%%%%%%%%%%%%%%%%%%%%%%%%%%%%%%%%%%%%%%%%%%%
\begin{definition}
For 
$\F(v_1, x_1; \ldots; v_{k}, x_k)  \in  C^{k}_{m}(V, \W)$,  and 
$\F(v'_{1}, y_{1}; \ldots; v'_{n}, y_{n})  \in   C_{m'}^{n}(V, \W)$ %$n'=n-k-1$, 
%%     
%%
 %with $r$ pairs of common vertex algebra formal parameters, 
%%
%let 
%%
%$(z_1, \ldots, z_{k+n})$ be the set of non-common formal parameters, and  
%%
%$(v_1, \ldots, v_k; v'_1, \ldots, v'_{n})$ be the set of  $V$-elements. 
%% 
%(and, correspondingly, $r$ formal variables), 
%%
%%
%Then 
%%
the product %$\F$
\begin{eqnarray}
\label{gendef}
&& \F(v_1, x_1; \ldots; v_{k}, x_k) \cdot_\epsilon \F(v'_{1}, y_{1}; \ldots; v'_{n}, y_{n})  
\nn
&&
\qquad \qquad \mapsto 
\widehat{R} \; \F\left( v_1, x_1; \ldots; v_{k}, x_k; v'_{1}, y_{1}; \ldots; v'_{n}, y_{n}; \epsilon\right), 
%%
%%\nn
%&&
%%
\end{eqnarray}
 is a $\W_{
%x_1, \ldots, x_k; y_{1}, \ldots, y_{n} 
%%
z_1, \ldots, z_{k+n-r}
 }$-valued rational 
form 
\begin{eqnarray}
\label{Z2n_pt_epsss}
&& \langle w',  \widehat{R} \; \F (v_1, x_1; \ldots; v_{k}, x_k; v'_{1}, y_{1}; \ldots; v'_{n}, y_{n}; \epsilon) \rangle 
%%
%\nn
%%
%&& \qquad =\langle w',  \F (
%v_1, x_1; \ldots; v_{k}, x_k) \cdot_{\epsilon} \F(v'_{1}, y_{1}; \ldots; v'_{n}, y_{n}
%
%v'_{1}, z_1;  \ldots; v'_{k+n'-r},  z_{k+n'-r}; 
%\epsilon
%) \rangle 
%%
\nn 
& & \quad  =  
\sum_{u\in V }  
 \langle w', Y^{W}_{WV}\left(  
\F (v_{1}, x_{1};  \ldots; v_{k}, x_{k}), \zeta_1\right)\; u \rangle  
\nn
& &
\quad   %\qquad  \qquad 
 \langle w', Y^{W}_{WV}\left( 
\F(v'_{1}, y_{1}; \ldots; {v'}_{i_1}, \widehat{y}_{i_1}; 
\ldots; \ldots; {v'}_{j_r}, \widehat{y}_{j_r}; \ldots;  v'_{n}, y_{n})
 , \zeta_{2}\right) \; \overline{u} \rangle,      
\end{eqnarray}
%%
%%
%defined by \eqref{Z2n_pt_eps1q1}.   
%%
%%
%%
via \eqref{def},  
parametrized by %a non-zero complex number $\epsilon$,  
%%
 %and 
%%
%complex numbers %two other complex parameters 
%%
$\zeta_1$, $\zeta_2  \in \C$, and we exclude all monomials $(x_{i_l} - y_{j_l})$, $1 \le l \le r$, from 
\eqref{Z2n_pt_epsss}. 
%% 
%characterizing \eqref{Z2n_pt_eps}, and .   
%%
%%
%%%%%%%%%%%%%%%%%%%%%%%%%%%%%%%%%%%%%%%%%%%%%%%%%%%%%%%
%%
%where here $u$ ranges over any $V_{l}$-basis and $\overline{u}$ is the dual of $u$ with respect to 
%%
The sum 
is taken over any $V_{l}$-basis $\{u\}$,  
where $\overline{u}$ is the dual of $u$ with respect to a non-degenerate bilinear form %the Li--Z metric
 $\langle .\ , . \rangle_\lambda$, \eqref{eq: inv bil form} over $V$, (see Appendix \ref{grading}). 
%%
  %_{\lambda}$,  %^{\mathrm{sq}}$  
%%
%%
%of \eqref{eq: inv bil form} as defined by the %square bracket 
%%
%% 
%%
%Virasoro operators $\{L[n]\}$. 
%%
%and with $\lambda$ of \eqref{eq:lamb_eps}.   
%%
\end{definition}
\begin{remark}
Due to the symmetry of the geometrical interpretation describe above, we could exclude 
from the set $(x_1, \ldots, x_k)$ in \eqref{Z2n_pt_epsss}  
$r$  
formal parameters which belong 
to coinciding pairs  resulting to the same 
definition of the $\epsilon$-product. 
\end{remark}
%%
%%%%%%%%%%%%%%%%%%%%%%%%%%%%%%%%%%%%%%%%%%%%%%%%%%%%%%%%%%%%%%%%%%%%%%%%%%%%%%%%%%%%%%%%%%%%%%%%%%%%
%Since the multipliers in the sum of \eqref{} are rational forms, 
%%
%%%%%%%%%%%%%%%%%%%%%%%%%%%%%%%%%%%%%%%%%%%%%%%%%%%%%%%%%%%%%%%%%%%%%%%%%%%%%%%%%%%%%%%%%%%%%%%%%%%%%%%%%%
%%%%%%%%%%%%%%%%%%%%%%%%%%%%%%%%%%%%%%%%%%%%%%%%%%%%%%%%%%%%%%%%%%%%%%%%%%%%%%%%%%%%%%%%%%%%%%%%%%%%%%%%%%
%\section{Product of spaces of $\W$-valued forms}
%%
%\label{product}
%%
%%%%%%%%%%%%%%%%%%%%%%%%%%%%%%%%%%%%%%%%%%%%%%%%%%%%%%%%%%%%%%%%%%%%%%%%%%%%%%%%%%%%%%%%%%%%%%%%%%%
%% 
%% 
%% 
%Due to \eqref{eq:lamb_eps}, 
%%
 %the product \eqref{Z2n_pt_eps} can be rewritten as
%%
%
%\bigskip 
%
%%%%
%
%where here $u$ ranges over any $V_{l}$-basis and $\overline{u}$ is the dual of $u$ with respect to 
%%
%the a %standard 
%%
%non-degenerate bilinear form %Li--Z metric  
%%
%$\langle u ,v\rangle_V$. %^{\mathrm{sq}}$. 
%%
By the standard reasoning \cite{FHL, Zhu}, %it is clear that %Definition 
 \eqref{Z2n_pt_epsss} does not depend on the choice of a basis of $u \in V_l$, $l \in \Z$.   
%%
%\end{remark}
%%
%%%%%%%%%%%%%%%%%%%%%%%%%%%%%%%%%%%%%%%%%%%%%%%%%%%%%%%%%%%%%%%%%%%%%%%%%%%%%%%%%%%%%%%%%%%%%%%%%%%%%%%%
%%
In the case when multiplied forms $\F$ do not contain $V$-elements, i.e., for $\Phi$, $\Psi \in \W$, 
  \eqref{Z2n_pt_epsss} defines the product $\Phi \cdot_\epsilon \Psi$ associated to 
 a %the
 rational function: 
\begin{eqnarray}
\label{Z2_part}
{\mathcal R(\epsilon)}= \sum_{l \in \Z }  \epsilon^l 
 \sum_{u\in V_l } 
\langle w', Y^{W}_{WV}\left(  
\Phi, \zeta_1\right) \; u \rangle 
\langle w', Y^{W}_{WV}\left(
\Psi,   
\zeta_2 \right) \; \overline{u} \rangle,  
\end{eqnarray}  
which defines $\F(\epsilon) \in \W$ via $\mathcal R(\epsilon)=\langle w', \F(\epsilon)\rangle$. 
%%
%%
%%%%%%%%%%%%%%%%%%%%%%%%%%%%%%%%%%%%%%%%%%%%%%%%%%%%%%%%%%%%%%%%%%%%%%%%%%%%%%%%%%%%%%%%%%%%%%%%%%%%%%%%%%%
%%%%%%%%%%%%%%%%%%%%%%%%%%%%%%%%%%%%%%%%%%%%%%%%%%%%%%%%%%%%%%%%%%%%%%%%%%%%%%%%%%%%%%%%%%%%%%%%%%%%%%%%%%%
\subsection{Convergence and properties of of the $\epsilon$-product}% and existence of corresponding rational form}
%%
%% 
%%%%%%%%%%%%%%%%%%%%%%%%%%%%%%%%%%%%%%%%%%%%%%%%%%%%%%%%%%%%%%%%%%%%%%%%%%%%%%%%%%%
%%%%%%%%%%%%%%%%%%%%%%%%%%%%%%%%%%%%%%%%%%%%%%%%%%%%%%%%%%%%%%%%%%%%%%%%%%%%%%%%%%%
 %In \cite{Zu}, 
%%
In order to prove convergence of a product of elements of two spaces
 $\W_{x_1, \ldots, x_k}$ 
and $\W_{y_1, \ldots, y_n}$  of rational $\W$-valued forms,   
we have to use a geometrical interpretation \cite{H2, Y}.
Recall that a $\W_{z_1, \ldots, z_n}$-space is defined by means of matrix elements of the form \eqref{def}. 
For a vertex algebra $V$, this corresponds \cite{FHL}  to a matrix element of a number of $V$-vertex operators
with formal parameters identified with local coordinates on a Riemann sphere. 
%%
%%%%%%%%%%%%%%%%%%%%%%%%%%%%%%%%%%%%%%%%%%%%%%%%%%%%%%%%%%%%%%%%%%%%%%%%%%%%%%%
% 
%% 
%
%
Geometrically, each space $\W_{z_1, \ldots, z_n}$ can be also associated to a Riemann sphere
 with a few marked points, %punctures 
and local coordinates 
%%
%(identified to $\zeta_1$, $\zeta_2$ as we will see in Subsection \ref{}) 
%%
 vanishing at these points %punctures 
\cite{H2}. 
An extra point %nother  of those points 
can be associated to a center of an annulus used in order 
to sew the sphere with another sphere.   
The product \eqref{Z2n_pt_epsss} has then a geometric interpretation.    
The resulting model space would also be associated to a Riemann sphere formed as a result of sewing procedure.  
%
% with three marked points. 
%%
%% should be parametrized by a nonzero complex number z or more generally,
%%
%in terms of spheres with three punctures and local coordinates (identified to $\zeta_1$, $\zeta_2$)
%%
 %vanishing at these punctures \cite{H2}. This will be explained in the proof of Proposition \ref{derga}. 
%%
%%%%%%%%%%%%%%%%%%%%%%%%%%%%%%%%%%%%%%%%%%%%%%%%%%%%%%%%%%%%%%%%%%%%%%%%%%%%%%%%%%%
%% 
%%
In Appendix \ref{sphere} we describe explicitly the geometrical procedure of sewing of two spheres \cite{Y}.  
%% 

%%%%%%%%%%%%%%%%%%%%%%%%%%%%%%%%%%%%%%%%%%%%%%%%%%%%%%%%%%%%%%%%%%%%%%%%%%%%%%%%%%%%%%%%%%%%%%%%%%%%%%%%
%% 
%%
Let us identify (as in \cite{H2, Y, Zhu, TUY, FMS, BZF}) two sets $(x_1, \ldots, x_k)$ and $(y_1, \ldots, y_n)$ of 
complex formal parameters,
with local
coordinates of two sets of points on the first and the second Riemann spheres correspondingly.   
Identify complex parameters $\zeta_1$, $\zeta_2$ of \eqref{Z2n_pt_epsss} with coordinates \eqref{disk} of 
the annuluses \eqref{zhopki}. 
%%
%Then let us involve the sewing procedure \cite{Y} of these spheres recalled in Appendix \ref{sphere}. 
%%
After identification of annuluses $\mathcal A_a$ and $\mathcal A_{\overline{a}}$, 
$r$ coinciding coordinates may occur. This takes into account case of coinciding formal parameters.
%%
%%%%%%%%%%%%%%%%%%%%%%%%%%%%%%%%%%%%%%%%%%%%%%%%%%%%%%%%%%%%%%%%%%%%%%%%%%%%%%%%%%%%%%%%%%%%%%%%%%%%%%%%%%%%%%%%%%%
In this way, we construct the map \eqref{gendef}.  
%% 

%%%%%%%%%%%%%%%%%%%%%%%%%%%%%%%%%%%%%%%%%%%%%%%%%%%%%%%%%%%%%%%%%%%%%%%%%%%%%%%%%%
%%%%%%%%%%%%%%%%%%%%%%%%%%%%%%%%%%%%%%%%%%%%%%%%%%%%%%%%%%%%%%%%%%%%%%%%%%%%%%%%%%
As we %will 
see in \eqref{Z2n_pt_epsss}, %in the next subsection, 
the product is defined 
by a sum of products of matrix elements  \cite{FHL} associated to each of two spheres.  
Such sum is supposed to describe a $\W$-valued rational differential form defined 
on a sphere formed %obtained 
as a result of geometrical sewing \cite{Y} of two initial spheres. 
Since two initial spaces $\W_{x_1, \ldots, x_k}$ and $\W_{y_1, \ldots, y_n}$ 
are defined through  rational-valued forms 
%%
%with the conditions of absolute convergence 
%% 
expressed by 
 matrix elements of the form \eqref{def}.   
We then arrive at 
%%
%it is then proved in \cite{Zu}, %(Proposition \ref{derga}),
%%
% that 
%%
the resulting product defines a 
$\W_{z_1, \ldots, z_{k+n-r} 
%%
%x_1, \ldots, x_k; y_1, \ldots, y_n
%%
%%z_1, \ldots, z_{k+n}
}$-valued rational 
form by means of an absolute convergent matrix element on the resulting sphere. 
%%
% formed from sewing of two initial spheres. 
%%
%In the next subsections we prove the existence of such rational form, and absolute convergence of 
%%
%corresponding matrix element.  
%
The complex sewing parameter, parameterizing the module space of sewin spheres, parametrizes also the product of 
$\W$-spaces.  
%

%%%%%%%%%%%%%%%%%%%%%%%%%%%%%%%%%%%%%%%%%%%%%%%%%%%%%%%%%%%%%%%%%%%%%%%%%%%%%%%%%
%%
Next, we formulate 
%%
%%%%%%%%%%%%%%%%%%%%%%%%%%%%%%%%%%%%%%%%%%%%%%%%%%%%%%%%%%%%%%%%%%%%%%%%%%%%%%%%%%%%%
%%%%%%%%%%%%%%%%%%%%%%%%%%%%%%%%%%%%%%%%%%%%%%%%%%%%%%%%%%%%%%%%%%%%%%%%%%%%%%%%%%%%%
\begin{definition}
\label{sprod}
We define the
%%
%Let us also define the 
%%
action of an element $\sigma \in S_{k+n-r}$ on the product of 
$\F  (v_{1}, x_{1};  \ldots; v_{k}, x_{k}) \in \W_{x_1, \ldots, x_k}$, and 
$\F  (v'_{1}, y_{1}; \ldots; v'_{n}, y_{n}) \in \W_{y_1, \ldots, y_n}$, as
%%
%is introduced as follows 
%%
%%%%%%%%%%%%%%%%%%%%%%%%%%%%%%%%%%%%%%%%%%%%%%%%%%%%%%%%%%%%%%%%%%%%%%%%%%%%%%%%%%%%%%%%%%%%%%%%%%%%%%%
\begin{eqnarray}
\label{Z2n_pt_epsss}
&& \langle w',  \sigma(\widehat{R}\;\F) (v_{1}, x_{1};  \ldots; v_{k}, x_{k}; v'_{1}, y_{1}; \ldots; v'_{n}, y_{n}; 
%%
%v'_{1}, z_1;  \ldots; v'_{k+n'-r},  z_{k+n'-r}; 
%%
\epsilon) \rangle 
\nn
&&
\qquad =\langle w',  \F (\widetilde{v}_{\sigma(1)}, z_{\sigma(1)};  \ldots; \widetilde{v}_{\sigma(k+n-r)},  z_{\sigma(k+n-r)};  
%%
%%
%%
%v_{\sigma(1)}, x_{\sigma(1)};  \ldots; v_{\sigma(k)}, x_{\sigma(k)}; 
%%
%v'_{\sigma(1)}, y_{\sigma(1)}; \ldots; v'_{\sigma(n)}, y_{\sigma(n)}; 
%%
%v'_{1}, z_1;  \ldots; v'_{k+n'-r},  z_{k+n'-r}; 
%%
\epsilon) \rangle 
%%%%%%%%%
\nn 
& & \qquad =  
\sum_{u\in V }  
 \langle w', Y^{W}_{WV}\left(  
\F (\widetilde{v}_{\sigma(1)}, z_{\sigma(1)};  \ldots; \widetilde{v}_{\sigma(k)}, z_{\sigma(k)}), \zeta_1\right)\; u \rangle  
\nn
& &
\qquad   %\qquad  \qquad 
 \langle w', Y^{W}_{WV}\left( 
\F
(\widetilde{v}_{\sigma(k+1)}, z_{\sigma(k+1)}; \ldots; \widetilde{v}_{\sigma(k+n-r)}, z_{\sigma(k+n-r)}) , \zeta_{2}\right) \; \overline{u} \rangle, 
\end{eqnarray}
where by $(\widetilde{v}_{\sigma(1)}, \ldots, \widetilde{v}_{\sigma(k+n-r)})$ we denote a permutation of 
\begin{equation}
\label{notari}
(\widetilde{v}_{1}, \ldots, \widetilde{v}_{k+n-r})
=(v_1, \ldots; v_k; \ldots, \widehat{v}'_{j_1}, \ldots, \widehat{v}'_{j_r},  \ldots  ). 
\end{equation}
\end{definition}
%%
%%%%%%%%%%%%%%%%%%%%%%%%%%%%%%%%%%%%%%%%%%%%%%%%%%%%%%%%%%%%%%%%%%%%%%%%%%%%%%%%%%%%%%%%%%%%%%%%%
%\begin{remark}
%%
Let $t$ be the number of common vertex operators the mappings  
$\F(v_{1}, x_{1}$;  $\ldots$; $v_{k}, x_{k}) \in C^{k}_{m}(V, \W)$ and 
$\F(v'_{1}, y_{1}; \ldots; v'_{n}, y_{n}) \in C^{n}_{m'}(V, \W)$, 
are composable with.  
%%
%Similar to the case of common formal parameters, this case is 
%% 
%separately treated with a decrease to $m+m'-t$ of number of composable vertex operators.   
%%
%In what follows, we exclude this case from considerations. 
%%
%\end{remark}
%%
%%
%%
%%%%%%%%%%%%%%%%%%%%%%%%%%%%%%%%%%%%%%%%%%%%%%%%%%%%%%%%%%%%%%%%%%%%%%%%%
%In \cite{Zu} 
%%
The rational 
%%In particular,  %%the proof of this proposition in \cite{Zu} 
%%
%%we show the convergence of the 
%%
form corresponding to the  
$\epsilon$-product $\widehat{R}  \F\left(v_1, x_1; \ldots; v_{k}, x_k; v'_{1}, y_{1}; \ldots; v'_{n}, y_{n}
; \epsilon\right)$ converges in $\epsilon$,  and %the fact it 
satisfies \eqref{shushu},  $L_V(0)$-conjugation \eqref{loconj} and 
$L_V(-1)$-derivative \eqref{lder1} properties. 
 Using Definition \ref{kakashka} of $C^n_m(V, \W)$-space and Definition \ref{mapco} of mappsings composable with 
vertex operators,  
%%
%we have proven 
%%
 we then have  
%%
%%%%%%%%%%%%%%%%%%%%%%%%%%%%%%%%%%%%%%%%%%%%%%%%%%%%%%%%%%%%%%%%%%%%%%%%%%%%%%5
\begin{proposition}
\label{tolsto}
For $\F(v_1, x_1; \ldots; v_{k}, x_k) \in C_{m}^{k}(V, \W)$ and 
$\F(v'_{1}, y_{1}; \ldots; v'_{n}, y_{n})\in C_{m'}^{n}(V, \W)$, 
the product $\widehat{R} \F\left(v_1, x_1; \ldots; v_{k}, x_k; v'_{1}, y_{1}; \ldots; v'_{n}, y_{n}
%%
%v_{1}, z_{1}; \ldots; v'_{k+n}, z_{{k+n}}
%%
; \epsilon\right)$ \eqref{Z2n_pt_epsss} 
belongs to the space $C^{k+n-r
}_{m+m'-t
}(V, \W)$, i.e.,  
\begin{equation}
\label{toporno}
\cdot_\epsilon : C^{k}_{m}(V,\W) \times C_{m'}^{n}(V, \W) \to  C_{m+m'-t
}^{k+n-r}(V, \W).  
\end{equation}
\hfill $\square$
\end{proposition}
%%
%%%%%%%%%%%%%%%%%%%%%%%%%%%%%%%%%%%%%%%%%%%%%%%%%%%%%%%%%%%%%%%%%%%%%%%%%%%%%%%%%%%%%%
%\begin{remark}
%%
%%
%\end{remark}
%%
%In this section we study %the
 %properties of 
%First, we have to define the action of various operators on 
%%
%the product 
%%
%$\F(v_1, x_1$; $ \ldots $; $  v_k, x_{k} $; $ v'_1, y_1; \ldots $; $  v'_{n},  y_{n} $; $ \epsilon )$ of 
%%
%\eqref{Z2n_pt_eps1q1}. 
%%
%%
%%%%%%%%%%%%%%%%%%%%%%%%%%%%%%%%%%%%%%%%%%%%%%%%%%%%%%%%%%%%%%%%%%%%%%%%%%%%%%%%%%%%%%%%%%%%%%%%%%%%%%%%%%%
\begin{remark}
Note that due to \eqref{wprop}, in 
Definition %\ref{duplodef}, and in 
\eqref{Z2n_pt_epsss} %in particular, 
it is assumed that 
$\F (v_{1}, x_{1} $; $  \ldots $ ; $  v_{k}, x_{k})$ and $\F(v'_{1}, y_1; \ldots; v'_{n}, y_{n})$
 are composable with the $V$-module $W$ vertex operators 
$Y_W(u, -\zeta_1)$ and $Y_W(\overline{u}, -\zeta_2)$ correspondingly.  
%%
%%(see Section \ref{application} for the definition 
%%
%%of composability).   
%%
The product \eqref{Z2n_pt_epsss} is actually defined by a sum of 
products of matrix elements of ordinary $V$-module $W$ vertex operators 
acting on $\W%_{z_1, \ldots, z_n}
$-elements.
%%
%According to \eqref{}, 
The elements 
%%
 %In what follows %Subsection \ref{}
%%
 %we will see that, since 
%%
$u \in V$ and $\overline{u} \in V'$ 
are connected  by \eqref{dubay}, and $\zeta_1$, $\zeta_2$ are related by \eqref{pinch}. 
%
%appear in a relation to each other. 
%%
%%Then the 
%%
%%the product \eqref{Z2n_pt_eps} depends on $\epsilon$.   
%
%of $\zeta_1\zeta_2$ is associated to $\epsilon$ via \eqref{pinch}, and, therefore, changes the power of $\epsilon$ in .  
%%
%%%%%%%%%%%%%%%%%%%%%%%%%%%%%%%%%%%%%%%%%%%%%%%%%%%%%%%%%%%%%%%%%%%%%%%%%%%%%%%%%%%%%%%%%%%%%%%%%%%%%%%%%%%%%%%%%%%
%%
The form %\eqref{} 
of the product defined above is natural in terms of the theory of chacaters for vertex operator algebras 
\cite{TUY, FMS, Zhu}.   
\end{remark}
%%%%%%%%%%%%%%%%%%%%%%%%%%%%%%%%%%%%%%%%%%%%%%%%%%%%%%%%%%%%%%%%%%%%%%%%%%%%%%%%%%%%%%%%%%%%%%%%%%%%%%%%%%
%%
\begin{remark}
%%
%In contrast to \cite{Zu}, 
%%
For purposes of construction of cohomological invariant, we do not exclude in this paper 
the case of $r$ pais of common formal parameters $x_i=y_j$, $1 \le i \le k$, $1 \le j \le n$, for 
$\F(v_1, x_1; \ldots; v_{k}, x_k) \in C_{m}^{k}(V, \W)$ and 
$\F(v'_{1}, y_{1}; \ldots; v'_{n}, y_{n})\in C_{m'}^{n}(V, \W)$  in Proposition \ref{}. 
Such formal parameter pairs are excluded from the right hand side of the map \eqref{toporno}. 
%%
%Since we assume that $(x_1, \ldots, x_k; y_{1}, \ldots, y_n)\in F_{k+n}\C$, i.e.,
%% coincidences of $x_i$ and $y_j$ are  %excluded by the definition of $F_{k+n}\C$.  
%%
\end{remark}
%%
%%%%%%%%%%%%%%%%%%%%%%%%%%%%%%%%%%%%%%%%%%%%%%%%%%%%%%%%%%%%%%%%%%%%%%%%%%%%%%%%%%%%%%%%%%
%Since the product of $\F(v_{1}, x_{1}$;  $\ldots$; v_{k}, x_{k}) \in C^{k}_{m}(V, \W)$ and 
%%
%$\F(v'_{1}, y_{1}; \ldots; v'_{n}, y_{n}) \in C^{n}_{m'}(V, \W)$ results in 
%%
%an element of $C^{k+n}_{m+m'}(V, \W)$, then,  
%%
%similar to Proposition \ref{tudaty} \cite{Huang}, the following 
%%%
%corollary follows directly from Proposition \eqref{tolsto} and Definition \ref{sprod}: %, \ref{}, and \ref{}:
%%%%%%%%%%%%%%%%%%%%%%%%%%%%%%%%%%%%%%%%%%%%%%%%%%%%%%%%%%%%%%%%%%%%%%%%%%%%%%%%%%%%%%%%%%%%%%%
%%
%%%%%%%%%%%%%%%%%%%%%%%%%%%%%%%%%%%%%%%%%%%%%%%%%%%%%%%%%%%%%%%%%%%%%%%%%%%%%%%%%%%%%%%%%%%%%%%%55
We then have 
%%
%In \cite{Zu} we also proved 
%%
two corollaries: 
\begin{corollary}%{\rm \cite{Huang}}  
For the spaces $\W_{x_1, \ldots, x_k}$ and $\W_{y_1, \ldots, y_n}$
 with the product \eqref{Z2n_pt_epsss} $\F \in \W_{z_1, \ldots, z_{k+n-r}  
%x_1, \ldots, x_k; y_1, \ldots, y_n
}$,   
the subspace of $\hom(V^{\otimes n}, 
%%
%\overline
%%
{\W}_{z_1, \ldots, z_{k+n-r}}$   
%
%x_1, \ldots, x_k; y_1, \ldots, y_n})$   
%%
consisting of linear maps 
having the $L_W(-1)$-derivative property, having the $L_V(0)$-conjugation property
or being composable with $m$ vertex operators is invariant under the 
action of $S_{k+n-r}$.
\end{corollary}
%%
%%%%%%%%%%%%%%%%%%%%%%%%%%%%%%%%%%%%%%%%%%%%%%%%%%%%%%%%%%%%%%%%%%%%%%%%%%%
%%%%%%%%%%%%%%%%%%%%%%%%%%%%%%%%%%%%%%%%%%%%%%%%%%%%%%%%%%%%%%%%%%%%%%%%%%%
%%
%Finally, we have the following
%%
\begin{corollary}
\label{functionformprop}
%%
%Let ${v}_{i}\in V$, $1 \le i \le k$, ${v'}_{j}\in V$, $1 \le j \le n$.   %and $v'_{j}\in V$, $1 \le j \le n$,
%%
% be vertex algebra elements. %quasi-primary vectors of weights %square bracket weight $wt[v_{i}]$ for $1 \le i \le k$,
%%
%${\rm wt}(v_{i})$ and  ${\rm wt}(v_{j})$.   % for $1 \le j \le n$, %and $j=1,\ldots, n$. 
%% 
%Let $x_{i}\in \widehat{\Sigma}^{(0)}_{1}$ and
%%
%$y_i \in \widehat{\Sigma}_{0}^{(2)}$  
%%
%with formal parameters related by the %sewing 
%%
%relation 
%%
%\[
%%
%x_i y_j=\epsilon=-\lambda^2, 
%%
%\]
%%
%for $1 \le i \le k$, $1 \le j \le n$, 
%%
%%
%and where the branch covering \eqref{eq:branch} is chosen with
%%
%\begin{equation*} 
%%
%\left( \frac{dy_{j}}{dx_{i}}\right)^{{\rm wt} (v_i)}=\left( \frac{\lambda}{x_{i}}\right)^{2{\rm wt} (v_i)}. 
%%
%\end{equation*} 
%%
%%
%Then 
%%
For a fixed set $(v_1, \ldots v_k; v_{k+1}, \ldots, v_{k+n}) \in V$ of vertex algebra elements, and
 fixed $k+n$, and $m+m'$,  
%%% 
 the $\epsilon$-product 
 $\F(v_1, z_1; \ldots; v_k, z_k; v_{k+1}, z_{k+1}; \ldots $ ; $ v_{k+n-r}, y_{k+n-r}; \epsilon)$,  
\[
\cdot_{\epsilon}: C^{k}_m(V, \W) \times C^{n}_{m'}(V, \W) \rightarrow C^{k+n-r}_{m+m'-t}(V, \W), 
\]
of the spaces $C^{k}_{m}(V, \W)$ and $C^{n}_{m'}(V, \W)$, 
for all choices 
of $k$, $n$, $m$, $m'\ge 0$, 
is the same element of $C^{k+n-r}_{m+m'-t}(V, \W)$
for all possible $k \ge 0$. 
%%
%  
%%
%%
%\end{eqnarray}
%%
%%
%%are maps to the same element of the space $C^{k+n}_{m+m'}(V, \W)$. %, %for all $0 \le l \le k$. 
%%
 %% independent of the choice of $l=0, \ldots, k+n$, 
% where $N_{k}$ is  
%the number of odd parity vectors in the set $\{v_1,\ldots ,v_{k}\}$ 
%%
\hfill $\square$
\end{corollary}
By Lemma \ref{popa}, elements of the space $C^{k+n-r}_{m+m'-t}$ resulting from the $\epsilon$-product are 
invariant with respect to changes of formal parameters of the group  
${\rm Aut}_{z_1, \ldots, z_{k+n-r}}\Oo^{(k+n-r)}$. 

%%
%%%%%%%%%%%%%%%%%%%%%%%%%%%%%%%%%%%%%%%%%%%%%%%%%%%%%%%%%%%%%%%%%%%%%%%%%%%%%%%%%%%%
We then have 
%%
%%%%%%%%%%%%%%%%%%%%%%%%%%%%%%%%%%%%%%%%%%%%%%%%%%%%%%%%%%%%%%%%%%%%%%%%%%%%%%%%%%%
%%%%%%%%%%%%%%%%%%%%%%%%%%%%%%%%%%%%%%%%%%%%%%%%%%%%%%%%%%%%%%%%%%%%%%%%%%%%%%
\begin{definition}
For fixed sets $(v_1, \ldots, v_k)$, $(v'_1, \ldots, v'_n) \in V$,  
$(x_1, \ldots, x_k)\in \C$, $(y_1, \ldots, y_n$) $\in \C$,   
 we call the set of all $\W_{x_1, \ldots, x_{k} ;   y_1, \ldots, y_{n}}$-valued rational forms 
$\widehat{R} \F(
v_1, x_1; \ldots;  v_k, x_{k} $ ; $ v'_1, y_1; \ldots;  v'_{n},  y_{n}; 
%%
%v_1, z_1; \ldots; v_n, z_n $; $ 
%%
\epsilon)$ defined  by \eqref{Z2n_pt_epsss} 
with the parameter $\epsilon$ exhausting all possible values,    
%%
%a whole module space $\mathcal M(\epsilon)$  
%%
 the complete product of the spaces $\W_{x_1, \ldots, x_k}$ and  $\W_{y_{1}, \ldots, y_n}$.  
\end{definition}
%%
%%%%%%%%%%%%%%%%%%%%%%%%%%%%%%%%%%%%%%%%%%%%%%%%%%%%%%%%%%%%%%%%%%%%%%%%%%%%%%%%%%%%%%%%
%%%%%%%%%%%%%%%%%%%%%%%%%%%%%%%%%%%%%%%%%%%%%%%%%%%%%%%%%%%%%%%%%%%%%%%%%%%%%%%%%%%%%%%%
%%
\subsection{Coboundary operator acting on the product space}
%%
%%%%%%%%%%%%%%%%%%
%%
%In this subsection we show that the product \eqref{Z2n_pt_epsss} admits the action of %corresponding 
%%
%the chain-cochain operator \eqref{hatdelta}. 
%%
%%%%%%%%%%%%%%%%%%%%%%%%%%%%%%%%%%%%%%%%%%%%%%%%%%%%%%%%%%%%%%%%%%%%%%%%%%%%%%%%%%%%%%%%
In Proposition \ref{tolsto} we proved that the  product \eqref{Z2n_pt_epsss} of elements   
   $\F_1 \C_{m}^{k}(V, \W)$ and $\F_2 \in C_{m'}^{n}(V, \W)$ belongs to $C^{k+n-r}_{m+m'-t}(V, \W)$.  
Thus, the product admits the action ot the differential operator $\delta^{k+n-r}_{m+m'-t}$ defined in 
\eqref{hatdelta} where $r$ is the number of common formal parameters, and $t$ the number of commpon 
composable vertex
operators for $\F_1$ and $\F_2$. 
%%  
%%%%%%%%%%%%%%%%%%%%%%%%%%%%%%%%%%%%%%%%%%%%%%%%%%%%%%%%%%%%%%%%%%%%%%%%%%%%%%%%%%%%%%%%
%%%%%%%%%%%%%%%%%%%%%%%%%%%%%%%%%%%%%%%%%%%%%%%%%%%%%%%%%%%%%%%%%%%%%%%%%%%%%%%%%%%%%%%%
%%
 The co-boundary operator \eqref{hatdelta}
%%
 % acting on the product of  
%%
%the spaces $C_{m}^{k}(V, \W)$ and $C_{m'}^{n'}(V, \W)$, 
%%
 %and \eqref{halfdelta}
%%
 possesses a variation of Leibniz law with respect to the product 
\eqref{Z2n_pt_epsss}. %In 
%%
%%
%Indeed, in 
%\cite{Zu} we proved  
%%
We then have
%%
%%%%%%%%%%%%%%%%%%%%%%%%%%%%%%%%%%%%%%%%%%%%%%%%%%%%%%%%%%%%%%%%%%%%%%%%%%%%%
%%
\begin{proposition}
\label{tosya}
For $\F(v_{1}, x_{1};  \ldots; v_{k}, x_{k}) \in C_{m}^{k}(V, \W)$ 
and 
$\F(v'_{1}, y_{1}; \ldots; v'_{n}, y_{n}) \in C_{m'}^{n}(V, \W)$, 
%%
% $w' \in W'$ one has 
%%
the action of $\delta_{m + m'-t}^{k + n-r}$ on their product \eqref{Z2n_pt_epsss} is given by 
\begin{eqnarray}
\label{leibniz}
%%
  % R\left( \langle w', 
%%
&& \delta_{m + m'-t}^{k + n-r} \left(  \F (v_{1}, x_{1};  \ldots; v_{k}, x_{k}) 
 \cdot_{\epsilon} \F (v'_{1}, y_{1}; \ldots; v'_{n}, y_{n}) \right) %\rangle \right) 
\nn
&&
 \qquad = 
%%
  % R\left( \langle w', 
%%
\left( \delta^{k}_{m} \F (\widetilde{v}_{1}, z_{1};  \ldots; \widetilde{v}_{k}, z_{k}) \right)  
%
%(v_{1}, x_{1};  \ldots; v_{k}, x_{k}) \right) 
%%
\cdot_{\epsilon} \F (\widetilde{v}_{k+1}, z_{k+1}; \ldots; \widetilde{v}_{k+n}, z_{k+n-r})  
\nn
&&
\; + (-1)^k 
%%
%R\left( \langle w', 
%%
\F (\widetilde{v}_{1}, z_{1};  \ldots; \widetilde{v}_{k}, z_{k}) \cdot_{\epsilon}   \left( \delta^{n-r}_{m'-t} 
\F(\widetilde{v}_{1}, z_{k+1}; \ldots; %v_{i_1}, y_{i_1}; \ldots; 
%%
%v'_{j_r}, y_{j_r}; \ldots; 
%%
 \widetilde{v}_{k+n-r}, z_{k+n-r})  \right),  %%\rangle \right).
\nn
&&
\end{eqnarray}
where we use the notation as in \eqref{zsto} and \eqref{notari}. 
\end{proposition}
Appendix \ref{duda} contains the proof of this Proposition.  
%%
%%%%%%%%%%%%%%%%%%%%%%%%%%%%%%%%%%%%%%%%%%%%%%%%%%%%%%%%%%%%%%%%%%%%%%
\begin{remark}
Checking %Definition 
\eqref{hatdelta} we see that an extra arbitrary vertex algebra element $v_{n+1} \in V$, as well as corresponding 
 extra arbitrary formal parameter $z_{n+1}$ appear as a result of the action of $\delta^{n}_m$ on 
$\F \in C^n_m(V, \W)$ mapping it to $C^{n+1}_{m-1}(V, \W)$. 
In application to the $\epsilon$-product \eqref{Z2n_pt_epsss} these extra arbitrary elements are involved in the 
definition of the action of $\delta_{m + m'-t}^{k + n-r}$ on 
 $\F (v_{1}, x_{1};  \ldots; v_{k}, x_{k}) 
 \cdot_{\epsilon} \F (v'_{1}, y_{1}; \ldots; v'_{n}, y_{n})$.  
\end{remark}
Note that both sides of \eqref{leibniz} belong to the space 
$C_{m + m'-t + 1}^{n + n' -r +1}(V, W)$. 
The co-boundary operators $\delta^n_m$ and $\delta^{n'}_{m'}$
 in \eqref{leibniz} do not include the number of common vertex algebra elements 
(and formal parameters), neither the number of common vertex operators corresponding mappings composable with.
 The dependence on common vertex algebra elements, parameters, and composable vertex operators is taken into 
account in mappings multiplying the action of co-boundary operators on $\Phi$.

%%%%%%%%%%%%%%%%%%%%%%%%%%%%%%%%%%%%%%%%%%%%%%%%%%%%%%%%%%%%%%%%%%%
%%%%%%%%%%%%%%%%%%%%%%%%%%%%%%%%%%%%%%%%%%%%%%%%%%%%%%%%%%%%%%%%%%%
%%
Finally, we have the following 
%%
%%%%%%%%%%%%%%%%%%%%%%%%%%%%%%%%%%%%%%%%%%%%%%%%%%%%%%%%%%%%%%%%%%%%%%%%%%%%%%
%%
\begin{corollary}
The multiplication \eqref{Z2n_pt_epsss} extends the chain-cochain 
complex  %\eqref{conde}--\eqref{hat-complex} 
structure of Proposition \ref{delta-square} to all products $C^k_m(V, \W) \times C^{n}_{m'}(V, \W)$, 
$k$, $n \ge0$, $m$, $m' \ge0$.  
%%
 %products of double complex spaces.    
%%
\hfill $\qed$
\end{corollary}
%%
%%%%%%%%%%%%%%%%%%%%%%%%%%%%%%%%%%%%%%%%%%%%%%%%%%%%%%%%%%%%%%%%%%%%%%%%%%%%%%%%%%%%%%%%%%%%%
%%%%%%%%%%%%%%%%%%%%%%%%%%%%%%%%%%%%%%%%%%%%%%%%%%%%%%%%%%%%%%%%%%%%%%%%%%%%%%%%%%%%%%%%%%%%%
%%
%%
%In this subsection we formulate another Lemma related to the product \eqref{}. 
%%
\begin{corollary}
The product \eqref{Z2n_pt_epsss} and the product operator \eqref{hatdelta}  
endow the space $C^k_m(V, \W)$ $\times$ $ C^n_m(V, \W)$, $k$, $n \ge0$, $m$, $m' \ge0$, %$n\ge 0$, $m\ge 0$
 with the structure of a %double
bi-graded differential algebra $\mathcal G(V, \W, \cdot_\epsilon, \delta^{k+n-r}_{m+m'-t})$. 
\hfill $\qed$
\end{corollary}
%%
%\bigskip 
%
%%%%%%%%%%%%%%%%%%%%%%%%%%%%%%%%%%%%%%%%%%%%%%%%%%%%%%%%%%%%%%%%%%%%%%%%%%%%%%%%%%%%%%%%%%%%%%
%%
For elements of the spaces $C^2_{ex}(V, \W)$ % of the exceptional double complex, and   
we have the following %\cite{Zu} %Proposition
%%
%%%%%%%%%%%%%%%%%%%%%%%%%%%%%%%%%%%%%%%%%%%%%%%%%%%%%%%%%%%%%%%%%%%%%%%%%%%%%%%%%%%%
\begin{corollary}
%%
 %The product of the elements of the spaces $C^n_m(V, \W)$ extends to the subspace $C^2_{ex}(V, \W)$ by 
%% 
%\eqref{}. 
%%
%It satisfies $\sigma$-symmetry \eqref{}, $L{(0)}$-conjugation \eqref{}, and $L{(-1)}$-derivative \eqref{} 
%%
%properties. It is invariant with respect to changes of ${\rm Aut}_{z_1, z_2, z_3} \; \Oo^{(3)}$ of formal parameters
%%
%%
The product of elements of the spaces $C^{2}_{ex} (V, \W)$ and  $C^n_{m} (V, \W)$ is given by 
\eqref{Z2n_pt_epsss}, %and 
%%
%%%%%%%%%%%%%%%%%%%%%%%%%%%%%%%%%%%%%%%%%%%%%%%%%%%%%%%%%%%%%%%%%%%%%%%%%%%%%%%%%%%% 
%%
\begin{equation}
\label{pupa3}
\cdot_\epsilon: C^{2}_{ex} (V, \W) \times C^n_{m} (V, \W) \to C^{n+2-r}_{m} (V, \W),  
\end{equation}
and, in particular, 
\[
\cdot_\epsilon: C^{2}_{ex} (V, \W) \times C^{2}_{ex} (V, \W) \to C^{4-r}_{0} (V, \W).  
\]
%%
%contains elements $\F$ composable with zero number of vertex operators.  
%%
%and 
%%
%% vse to zhe samoe, no usloviya na \F drugie v smysle composable 
\hfill $\square$
\end{corollary}
%%
%%
%%
%%%%%%%%%%%%%%%%%%%%%%%%%%%%%%%%%%%%%%%%%%%%%%%%%%%%%%%%%%%%%%%%%%%%%%%%%%%%%%%%%%%%%%%%%%%%%%%%%%%%%%%%%%%
%%%%%%%%%%%%%%%%%%%%%%%%%%%%%%%%%%%%%%%%%%%%%%%%%%%%%%%%%%%%%%%%%%%%%%%%%%%%%%%%%%%%%%%%%%%%%%%%%%%%%%%%%%%
\subsection{The commutator} 
%
%%%%%%%%%%%%%%%%%%%%%%%%%%%%%%%%%%%%%%%%%%%%%%%%%%%%%%%%%%%%%%%%%%%%%%%%%%%%%%%%%%%%%%%%%%%%%%%%%%%%%%%%
%%
Let us consider the mappings 
%%
%\[
$\Phi(v_1, z_1 $ ; $  \ldots $; $ v_{n}, z_k)$ $\in$   $C_{m}^{k}(V, \W)$, and  
%\] 
%%
%%
%\[
$\Psi (v_{k+1}, z_{k+1}; % \otimes 
\ldots; 
%%
 %5\otimes
%%
 v_{k+n}, z_{k+n}) \in   C_{m'}^{n}(V, \W)$,  
%\]     
%%
 with have $r$ common vertex algebra elements (and, correspondingly, $r$ formal variables), and 
 $t$ common vertex operators mappings $\Phi$ and $\Psi$ are composable with. 
Note that when applying the co-boundary operators \eqref{marsha} and \eqref{halfdelta1} to a map 
$\Phi(v_1, z_1; %\otimes
 \ldots; % \otimes 
v_n, z_n) %(z_1, \ldots, z_n) 
\in C^n_m(V, \W)$, 
\[
\delta^n_m: \Phi(v_1, z_1; %\otimes 
\ldots; % \otimes
 v_n,z_n)%(z_1, \ldots, z_n) 
\to 
\Phi(v'_1, z'_1; %\otimes
 \ldots; %\otimes 
v'_{n+1}, z'_{n+1}) %(z'_1, \ldots, z'_{n+1}) 
\in C^{n+1}_{m-1}(V, \W),
\]
 one does not necessary assume that we keep 
 the same set of vertex algebra elements/formal parameters and 
vertex operators composable with for $\delta^n_m \Phi$, 
though it might happen that some of them could be common with $\Phi$.    
%%

%%%%%%%%%%%%%%%%%%%%%%%%%%%%%%%%%%%%%%%%%%%%%%%%%%%%%%%%%%%%%%%%%%%%%%%%%%%%%%%%%%%%%%%%%%%%%%%%
%%
Let us define an extra product (related to the $\epsilon$-product) 
the product of $\Phi$ and $\Psi$, %naturally coming from the definition of the spaces  
%%
%$C_{m}^{n}(V, W)$,
%5
 %as 
%%
 \begin{eqnarray}
%%
%\label{defproduct}
%%
&&
 \Phi \cdot \Psi: V^{\otimes(k +n-r)} \to  \W_{z_1, \ldots, z_{k+ n-r}}, \; 
\\
\label{product}
&&
 \Phi \cdot \Psi = \left[\Phi,_{\cdot \epsilon} \Psi\right]= \Phi \cdot_\epsilon \Psi- \Psi \cdot_\epsilon \Phi,    
\end{eqnarray}
where brackets denote ordinary commutator in $\W_{z_1, \ldots, z_{k+ n-r}}$.  
%%
%%%%%%%%%%%%%%%%%%%%%%%%%%%%%%%%%%%%%%%%%%%%%%%%%%%%%%%%%%%%%%%%%%%%%%%%%%%%%%%%%%%%%%%%%%%%%%%
Due to the properties of the maps $\Phi\in C_{m}^{k}(V, \W)$ and 
$\Psi\in   C_{m'}^{n}(V, \W)$, the map $(\Phi \cdot_\epsilon \Psi)$
belongs to the space $C_{m + m'- t }^{k +n-r}(V, \W)$.  
For $k=n$ and 
\[
\Psi (v_{n+1}, z_{n+1}; %\otimes
 \ldots; %\otimes 
v_{ 2n}, z_{2n}) %(z_{n+1}, \ldots, z_{2n}) 
= 
\Phi(v_{1}, z_1; %\otimes
 \ldots; %\otimes 
v_{ n}, z_n), %(z_1, \ldots, z_{n}),  
\]
we obtain from \eqref{product} and \eqref{Z2n_pt_epsss} that 
\begin{eqnarray}
\label{fifi}
\Phi(v_{1}, z_1; %\otimes
 \ldots;  %\otimes 
v_{ n}, z_n) %(z_1, \ldots, z_{n}) 
\cdot
 \Phi(v_{1}, z_1; %\otimes 
\ldots;  %\otimes
 v_{ n}, z_n) %(z_1, \ldots, z_{n})
=0. 
\end{eqnarray}  
%%
%%%%%%%%%%%%%%%%%%%%%%%%%%%%%%%%%%%%%%%%%%%%%%%%%%%%%%%%%%%%%%%%%%%%%%%%%%%%%%%%%%%%%%%%%%%%%%%%%%%%
%%%%%%%%%%%%%%%%%%%%%%%%%%%%%%%%%%%%%%%%%%%%%%%%%%%%%%%%%%%%%%%%%%%%%%%%%%%%%%%%%%%%%%%%%%%%%%%%%%%%
\section{The invariants}
\label{invariants}
In this section we provide the main result of the paper by deriving 
the simplest cohomological invariants associated to the short double complex \eqref{halfdelta} 
for a grading-restricted vertex algebra.   

Let us give first some further definitions. 
In this section we skip the dependence on vertex algebra elements and formal parameters in notations for 
 elements of $C_{n}^{m}(V, \W)$.  
%%
%%%%%%%%%%%%%%%%%%%%%%%%%%%%%%%%%%%%%%%%%%%%%%%%%%%%%%%%%%%%%%%%%%%%%%%%%%%%%%%%%%%%%%%%%%%
\begin{definition}
In analogy with differential forms, we call a map 
$\Phi \in   C_{m}^{n}(V, \W)$ closed if
\[
\delta^{n}_{m} \Phi=0.
\]  
For $m \ge 1$, we call it exact if there exists $\Psi \in C_{m-1}^{n+1}(V, \W)$ 
such that 
\[
\Psi=\delta^{n}_{m} \Phi.
\]   
\end{definition}
\begin{definition}
For $\Phi \in  {C}^{n}_{m}(V, \W)$ we call the cohomology class of mappings 
 $\left[ \Phi \right]$ the set of all closed forms that differs from $\Phi$ by an 
exact mapping, i.e., for $\chi \in  {C}^{n-1}_{m+1}$, 
\[
\left[ \Phi \right]= \Phi + \delta^{n-1}_{m+1} \chi, 
\]
 (we assume that both parts of the last formula belongs to the same space ${C}^{n}_{m}(V, \W)$).   
\end{definition}
Under a natural extra condition, the short double complex \eqref{shortseq} allows us to establish relations 
among elements of double complex spaces. 
In particular, we require that for a pair of double complex spaces $C^{n_1}_{k_1}(V, \W)$ and  
$C^{n_2}_{k_2}(V, \W)$ there exist subspaces 
${C'}^{n_1}_{k_1}(V, \W) \subset C^{n_1}_{k_1}(V, \W)$ and 
${C'}^{n_2}_{k_2}(V, \W) \subset C^{n_2}_{k_2}(V, \W)$ such that for $\Phi_1 \in {C'}^{n_1}_{k_1}(V, \W)$ and 
$\Phi_2 \in C'^{n_2}_{k_2}(V, \W)$, 
\begin{equation}
\label{ortho}
\Phi_1 \cdot \delta^{n_2}_{k_2} \Phi_2=0,  
\end{equation}
 namely,  
$\Phi_1$ supposed to be orthogonal to $\delta^{n_2}_{k_2} \Phi_2$ 
(i.e., commutative with respect to the product \eqref{product}). 
We call this {\it  the orthogonality condition} for 
mappings and actions of co-boundary operators for a double complex. 
It is easy to see that  
the assumption to belong to the same double complex space for both sides of the equations following from the orthogonality
 condition applies the bi-grading condition on double complex spaces. 
Note that in the case of differential forms considered on a smooth manifold, 
the Frobenius theorem for a distribution provides the orthogonality condition.   
In this Section we derive algebraic relations occurring from the orthogonality condition on 
the short double complex \eqref{shortseq}. 
We formulate 
%%
%%%%%%%%%%%%%%%%%%%%%%%%%%%%%%%%%%%%%%%%%%%%%%%%%%%%%%%%%%%%%%%%%%%%%%%%%%%%%%%%%%%%%%%%%%%%%%%%%%%%%%%%%%%%%%
\begin{proposition}
The orthogonality condition for the short double complex sequence \eqref{shortseq} determines the cohomological classes: 
\begin{equation}
\label{stupor}
\left[\left(\delta^{1}_{2} \Phi \right)\cdot \Phi \right], \; 
\left[\left(\delta^{0}_{3} \chi \right)\cdot \chi \right], \;  
\left[\left(\delta^{1}_{t} \alpha \right)\cdot \alpha \right],
\end{equation}
 for $0 \le t \le 2$, with 
non-vanishing $\left(\delta^{1}_{2} \Phi \right)\cdot \Phi$,
$\left(\delta^{0}_{3} \chi \right)\cdot \chi$, and 
$\left(\delta^{1}_{t} \alpha \right)\cdot \alpha$.  
 These classes are independent on the choice of $\Phi \in C^{1}_{2}(V, \W)$, $\chi \in C^{0}_{3}(V, \W)$, and
  $\alpha \in C^1_t(V, \W)$. 
\end{proposition}
\begin{remark}
 A cohomology class with vanishing $\left(\delta^{1}_{2} \Phi \right)\cdot \Phi  \cdot \alpha$ is given by 
$\left[ \left(\delta^{1}_{2} \Phi \right)\cdot \Phi  \cdot \alpha \right]$. 
\end{remark}
\begin{proof}
%%
%%%%%%%%%%%%%%%%%%%%%%%%%%%%%%%%%%%%%%%%%%%%%%%%%%%%%%%%%%%%%%%%%%%%%%%%%%%%%%%%%%%%%%%%%%%%%%%%%%%%%%%%%%%%%%
%%
 Let us consider two maps $\chi \in C^{0}_{3}(V, \W)$, $\Phi \in C^{1}_{2}(V, \W)$.
We require them to be orthogonal, i.e., 
\begin{equation}
\label{isxco}
\Phi \cdot \delta^{0}_{3} \chi=0.   
\end{equation}
Thus, there exists $\alpha \in C^{n}_{m}(V, \W)$, such that 
\begin{equation}
\label{uravnenie}
\delta^{0}_{3} \chi= \Phi \cdot \alpha, 
\end{equation} 
and 
$1=1 + n -r$, $2=2+m-t$, i.e., $n=r$,  which leads to $r=1$;  $m=t$, $0\le t \le 2$, i.e., 
$\alpha \in C^{1}_{t}(V, \W)$. 
%%%%%%%%
%%
All other orthogonality conditions for the short sequence \eqref{shortseq} does not allow relations of the form
 \eqref{uravnenie}. 

%%%%%%%%%%%%%%%%%%%%%%%%%%%%%%%%%%%%%%%%%%%%%%%%%%%%%%%%%%%%%%%%%%%%%%%%%%%%%%%%%%%%%%%%%%%%%%%%%%%%%%%%%%%%%%
%%
Consider now \eqref{isxco}.
We obtain, using \eqref{leibniz}
\[
 \delta^{2-r'}_{4-t'} (\Phi \cdot \delta^{0}_{3} \chi)=  
\left(\delta^{1}_{2} \Phi\right) \cdot \delta^{0}_{3} \chi + \Phi \cdot \delta^{1}_{2} \delta^{0}_{3} \chi=
 \left(\delta^{1}_{2} \Phi\right) \cdot \delta^{0}_{3} \chi= \left(\delta^{1}_{2} \Phi\right) \cdot \Phi \cdot 
\alpha. 
\]
Thus 
\[
 0=\delta^{3-r'}_{3-t'} \delta^{2-r'}_{4-t'} (\Phi \cdot \delta^{0}_{3} \chi)=  
\delta^{3-r'}_{3-t'}  \left( \left(\delta^{1}_{2} \Phi\right) \cdot \Phi \cdot 
\alpha.  \right), 
\]
and $\left(\left(\delta^{1}_{2} \Phi\right) \cdot \Phi \cdot \alpha\right) $ is closed. 
At the same time, 
from \eqref{isxco} %$\Phi \cdot \delta^{0}_{3} \chi=0$,
 it follows that 
\[
0=\delta^{1}_{2} \Phi \cdot \delta^{0}_{3} \chi- \Phi \cdot\delta^{1}_{2}\delta^{0}_{3} \chi 
= \left( \Phi\cdot \delta^{0}_{3} \chi \right).
\] 
Thus 
\[
\delta^{1}_{2} \Phi \cdot \delta^{0}_{3} \chi= \delta^{1}_{2} \Phi \cdot \Phi \cdot \alpha =0.
\] 
Consider \eqref{uravnenie}. 
Acting by $\delta^{1}_{2}$ and substituting back we obtain 
\[
0= \delta^{1}_{2} \delta^{0}_{3} \chi= \delta^{1}_{2}(\Phi \cdot \alpha)=  
\delta^{1}_{2}(\Phi) \cdot \alpha - \Phi \cdot \delta^{1}_{t} \alpha. 
\]
thus 
\[
\delta^{1}_{2}(\Phi) \cdot \alpha = \Phi \cdot \delta^{1}_{t} \alpha. 
\]
The last equality trivializes on applying $\delta^{3}_{t+1}$ to both sides. 
%%

%%%%%%%%%%%%%%%%%%%%%%%%%%%%%%%%%%%%%%%%%%%%%%%%%%%%%%%%%%%%%%%%%%%%%%%%%%%%%%%%%%%%%%%%%%%%%%%%%%%%%%%%%%%%
Let us show now the non-vanishing property of $\left(\left(\delta^{1}_{2} \Phi \right)\cdot \Phi\right)$. 
Indeed, suppose $\left(\delta^{1}_{2} \Phi \right)\cdot \Phi=0$. Then there exists $\gamma \in C^{n}_{m}(V, \W)$, 
such that $\delta^{1}_{2} \Phi =\gamma \cdot \Phi$. Both sides of the last equality should belong to the same double complex 
space but one can see that it is not possible. 
Thus, $\left(\delta^{1}_{2} \Phi \right)\cdot \Phi$ is non-vanishing. 
One proves in the same way that $\left(\delta^{0}_{3} \chi \right)\cdot \chi$ and 
$\left(\delta^{1}_{t} \alpha \right)\cdot \alpha$ do not vanish too.  
Now let us show  that $\left[\left(\delta^{1}_{2} \Phi \right)\cdot \Phi \right]$ 
 is invariant, i.e., it does not depend on the choice of $\Phi \in C^1_2(V, \W)$.
 Substitute $\Phi$ by 
$\left(\Phi + \eta\right)\in C^{1}_{2}(V, \W)$.   
We have 
\begin{eqnarray}
\label{pokaz}
\nonumber
\left(\delta^{1}_{2} \left( \Phi + \eta \right) \right) \cdot \left( \Phi  + \eta \right) &=& 
\left(\delta^{1}_{2} \Phi\right) \cdot \Phi 
 + \left( \left(\delta^{1}_{2} \Phi \right)\cdot \eta 
-  \Phi \cdot \delta^{1}_{2} \eta  \right)  
\nn
&+& \left( \Phi \cdot \delta^{1}_{2} \eta   + 
\delta^{1}_{2} \eta  \cdot \Phi \right) 
 +
\left(\delta^{1}_{2} \eta \right) \cdot \eta. 
\end{eqnarray}
Since
\[
\left( \Phi \cdot \delta^{1}_{2} \eta  + 
\left(\delta^{1}_{2} \eta\right)  \cdot \Phi \right)= 
\Phi \delta^{1}_{2} \eta -  (\delta^{1}_{2} \eta) \Phi 
+ \left(\delta^{1}_{2} \eta\right) \Phi - \Phi \; \delta^{1}_{2} \eta=0, 
\]
then \eqref{pokaz} represents the same cohomology class 
$\left[ \left(\delta^{1}_{2}  \Phi \right) \cdot \Phi \cdot \alpha \right]$. 
The same folds for $\left[\left(\delta^{0}_{3} \chi \right)\cdot \chi \right]$, and 
$\left[\left(\delta^{1}_{t} \alpha \right)\cdot \alpha \right]$. 
\end{proof}
%%%%%%%%%%%%%%%%%%%%%%%%%%%%%%%%%%%%%%%%%%%%%%%%%%%%%%%%%%%%%%%%%%%%%%%%%%%%%%%%%%%%%%%%%%%%
%%
\begin{remark}
Due to Proposition \ref{}, all chahomological classes are invariant with respect to correponding group 
${\rm Aut}_{z_1, \ldots, z_n}\Oo^{(n)}$ 
changes of formal parameters. 
\end{remark}
%%
%%%%%%%%%%%%%%%%%%%%%%%%%%%%%%%%%%%%%%%%%%%%%%%%%%%%%%%%%%%%%%%%%%%%%%%%%%%%%%%%%%%%%%%%%%
 The orthogonality condition for a double complex sequence \eqref{shortseq}, 
 together with the action of co-boundary operators 
\eqref{hatdelta} and \eqref{halfdelta}, and the multiplication formulas \eqref{product}--\eqref{leibniz},  
define a differential bi-graded algebra depending on vertex algebra elements and formal parameters. 
In particular, for the short sequence \eqref{shortseq}, we obtain in this way the generators and commutation relations 
for a continual Lie algebra $\mathcal G(V)$ (a generalization of ordinary Lie algebras with continual  
space of roots, c.f. \cite{saver})  
with the continual root space represented by a grading-restricted vertex algebra $V$. 
\begin{lemma}
For the short sequence \eqref{shortseq} we get a continual Lie algebra $\mathcal G(V)$
with generators 
\begin{equation}
\label{generators}
\left\{\Phi(v_1),\; \chi, \;\alpha(v_2),\; \delta^1_2 \Phi(v_1),\; \delta^0_3 \chi,\; \delta^1_t \alpha (v_2), 
0 \le t \le 2
\right\}, 
\end{equation}
 and commutation relations for a continual Lie algebra $\mathcal G(V)$ 
\begin{eqnarray}
\label{comid}
%%
%\Phi\cdot \delta^1_2\Phi= \left(\delta^1_2\Phi\right) \cdot  \Phi; \; \;  &&
%%
 %% \chi\cdot \delta^0_3\chi &=& - \left(\delta^0_3\chi\right) \cdot \chi;
  \; \;  
%%
 %\alpha \cdot \delta^1_t\alpha= \left(\delta^1_t\alpha\right) \cdot  \alpha;  
%%
%\nn
%%
\Phi\cdot \delta^1_t \alpha&=&\alpha \cdot \delta^1_2 \Phi  \ne 0,
\nn
% \; \; 
%%
  \delta^{0}_{3} \chi&=& \Phi \cdot \alpha,
%%\nn
%%  
%%
\end{eqnarray}  
with all other relations being trivial.  
The sum of cohomological classes \eqref{stupor} provides an invariant of $\mathcal G(V)$. 
\end{lemma}
\begin{proof}
Recall that $\Phi(v_1)(z_1) \in C^1_2(V, \W)$, $\chi \in C^0_3(V, \W)$, $\alpha \in C^1_t(V, \W)$, $0 \le t \le 2$. 
One easily checks the commutation relations coming from the orthogonality and bi-grading conditions. 
Further applications of \eqref{hatdelta}, \eqref{halfdelta}, and \eqref{ortho} to \eqref{shortseq}
 lead to trivial results. 
$\Phi \cdot \delta^1_t \alpha\ne 0$ is proven by contradiction.  
It is easy to check Jacobi identities for \eqref{generators} and \eqref{comid}. 
With a redefinition  
\begin{eqnarray}
\label{redifinition}
H&=&\delta^0_3\chi, 
\nn
%%\; \; 
H^*&=&\chi, 
\nn
%% \; \;
X_+(v_1)&=&\Phi(v_1),
\nn % \; \; 
X_-(v_2)&=&\alpha(v_2), 
\nn
 Y_+(v_1)&=&\delta^1_2\Phi(v_1),
\nn
 %% \; \; 
 Y_-(v_2)&=&\delta^1_t\alpha(v_2), 
\end{eqnarray}
the commutation relations \eqref{comid} become:
\begin{eqnarray*}
\left[X_+(v_1), X_-(v_2) \right]&=&H,
\nn %% \; \; 
\left[X_+(v_1), Y_-(v_1) \right] &=& \left[X_-(v_2), Y_+(v_1) \right], 
\end{eqnarray*}
i.e., the orthogonality condition brings about a representation of an affinization \cite{K} of continual counterpart  
of the Lie algebra $sl_2$.  
 Vertex algebra elements in \eqref{redifinition} play the role of 
roots belonging to continual non-commutative root space given by a vertex algebra $V$.  
\end{proof}
%%
%%%%%%%%%%%%%%%%%%%%%%%%%%%%%%%%%%%%%%%%%%%%%%%%%%%%%%%%%%%%%%%%%%%%%%%%%%%%%%%%%%%%%%%%%%%%%%
\section*{Acknowledgments}
The author would like to thank 
Y.-Z. Huang,  H. V. L\^e, and P. Somberg 
for related discussions. 
Research of the author was supported by the GACR project 18-00496S and RVO: 67985840. 
%%
%%%%%%%%%%%%%%%%%%%%%%%%%%%%%%%%%%%%%%%%%%%%%%%%%%%%%%%%%%%%%%%%%%%%%%%%%%%%%%%%%%%%%%%%%%
%%%%%%%%%%%%%%%%%%%%%%%%%%%%%%%%%%%%%%%%%%%%%%%%%%%%%%%%%%%%%%%%%%%%%%%%%%%%%%%%%%%%%%%%%%%
\section{Appendix: Grading-restricted vertex algebras and their modules}
\label{grading}
In this section, following \cite{Huang} we recall basic properties of 
grading-restricted vertex algebras and their grading-restricted generalized 
modules, useful for our purposes in later sections. 
We work over the base field $\C$ of complex numbers. 
%%
%%%%%%%%%%%%%%%%%%%%%%%%%%%%%%%%%%%%%%%%%%%%%%%%%%%%%%%%%%%%%%%%%%%%%%%%%%%%%%%%%%%%%%%%%%5
%%
%%
A vertex algebra  
$(V,Y_V,\mathbf{1})$, cf. \cite{K},  consists of a $\Z$-graded complex vector space 
\[
V = %\coprod
\bigoplus_{n\in\Z}\,V_{(n)}, \quad \dim V_{(n)}<\infty\,\, \mbox{for each}\,\, n\in \Z, 
\]
and linear map 
\[
Y_V:V\rightarrow {\rm End \;}(V)[[z,z^{-1}]], 
\]
 for a formal parameter $z$ and a 
distinguished vector $\mathbf{1_V}\in V$.  
The evaluation of $Y_V$ on $v\in V$ is the vertex operator
\[
Y_V(v)\equiv Y_V(v,z) = \sum_{n\in\Z}v(n)z^{-n-1}, 
\]
with components
\[
(Y_{V}(v))_{n}=v(n)\in {\rm End \;}(V),
\]
 where $Y_V(v,z)\mathbf{1} = v+O(z)$.
Now we describe further restrictions \cite{Huang}, defining a grading-restricted vertex algebra: 
\noindent
\begin{enumerate}
\item 
\noindent
{Grading-restriction condition}:
$V_{(n)}$ is finite dimensional for all $n\in \Z$, and $V_{(n)}=0$ for $n\ll 0$. 

\item { Lower-truncation condition}:
For $u, v\in V$, $Y_{V}(u, z)v$ contains only finitely many negative 
power terms, that is, $Y_{V}(u, z)v\in V((z))$ (the space of formal 
Laurent series in $z$ with coefficients in $V$).   

\item { Identity property}: 
Let $\one_{V}$ be the identity operator on $V$. Then 
\[
Y_{V}(\mathbf{1}_V, z)={\rm Id}_{V}.
\] 

\item { Creation property}: For $u\in V$, $Y_{V}(u, z)\mathbf{1}_V\in V[[z]]$
and 
\[
\lim_{z\to 0}Y_{V}(u, z)\mathbf{1}_V=u.
\]

\item { Duality}: For $u_{1}, u_{2}, v\in V$, 
$v'\in V'=\coprod_{n\in \mathbb{Z}}V_{(n)}^{*}$ ($V_{(n)}^{*}$ denotes
the dual vector space to $V_{(n)}$ and $\langle\,. ,. \rangle$ the evaluation 
pairing $V'\otimes V\to \C$), the series 
%%
%%\begin{eqnarray*}
%%
%%& & $\langle v', Y_{V}(u_{1}, z_{1})Y_{V}(u_{2}, z_{2})v\rangle$, %\quad
%%
$\langle v', Y_{V}(u_{2}, z_{2})Y_{V}(u_{1}, z_{1})v\rangle$, and % \quad\mbox{and} \\
%& & 
$\langle v', Y_{V}(Y_{V}(u_{1}, z_{1}-z_{2})u_{2}, z_{2})v\rangle$, 
%%\end{eqnarray*}
%%
are absolutely convergent
in the regions $|z_{1}|>|z_{2}|>0$, $|z_{2}|>|z_{1}|>0$,
$|z_{2}|>|z_{1}-z_{2}|>0$, respectively, to a common rational function 
in $z_{1}$ and $z_{2}$ with the only possible poles at $z_{1}=0=z_{2}$ and 
$z_{1}=z_{2}$. 

%%%%%%%%%%%%%%%%%%%%%%%%%%%%%%%%%%%%%%%%%%%%%%%%%%%%%%%%%%%%%%%%%%%%%%%%%%%%%%%%%%%%
One assumes the existence of Virasoro vector $\omega\in V$:
its vertex operator 
$Y(\omega, z)=\sum_{n\in\Z}L(n)z^{-n-2}$
is determined by Virasoro operators $L(n): V\to V$ fulfilling 
(notice that with abuse of notation we denote $L_V(n)=L(n)$)
\[
[L(m), L(n)]=(m-n)L(m+n)+\frac{c}{12}(m^{3}-m)\delta_{m+b, 0}{\rm Id_V},
\]
($c$ is called the {\it central charge} of $V$).  
The grading operator is given by $L(0)u=nu,\quad u\in V_{(n)}$, 
($n$ is called the weight of $u$ and denoted by $\wt(u)$).
\item { $L_V(0)$-bracket formula}: Let $L_{V}(0): V\to V$ 
be defined by $L_{V}(0)v=nv$ for $v\in V_{(n)}$. Then
\[
[L_{V}(0), Y_{V}(v, z)]=Y_{V}(L_{V}(0)v, z)+z\frac{d}{dz}Y_{V}(v, z),
\] 
for $v\in V$.
\item { $L_V(-1)$-derivative property}: 
Let $L_{V}(-1): V\to V$ be the operator 
given by 
\[
L_{V}(-1)v=\res_{z}z^{-2}Y_{V}(v, z)\one=Y_{(-2)}(v)\one,
\] 
for $v\in V$. Then for $v\in V$, 
\begin{equation*}
\label{derprop}
\frac{d}{dz}Y_{V}(u, z)=Y_{V}(L_{V}(-1)u, z)=[L_{V}(-1), Y_{V}(u, z)].
\end{equation*}
\end{enumerate}
%%
%%%%%%%
Correspondingly, a {grading-restricted generalized $V$-module} is a vector space 
$W$ equipped with a vertex operator map 
\[
Y_{W}: V\otimes W \to W[[z, z^{-1}]],
\] 
\begin{eqnarray*}
u\otimes w&\mapsto & Y_{W}(u, w)\equiv Y_{W}(u, z)w=\sum_{n\in \Z}(Y_{W})_{n}(u,w)z^{-n-1}, 
\end{eqnarray*}
and linear operators $L_{W}(0)$ and $L_{W}(-1)$ on $W$ satisfying conditions similar as in the 
definition for a grading-restricted vertex algebra. In particular, 
\begin{enumerate}
\item {Grading-restriction condition}:
The vector space $W$ is $\mathbb C$-graded, that is, 
$W=\coprod_{\alpha\in \mathbb{C}}W_{(\alpha)}$, such that 
$W_{(\alpha)}=0$ when the real part of $\alpha$ is sufficiently negative. 

\item { Lower-truncation condition}:
For $u\in V$ and $w\in W$, $Y_{W}(u, z)w$ contains only finitely many negative 
power terms, that is, $Y_{W}(u, z)w\in W((z))$.

\item { Identity property}: 
Let ${\rm Id}_{W}$ be the identity operator on $W$,
$Y_{W}(\mathbf{1}, z)={\rm Id}_{W}$.

\item { Duality}: For $u_{1}$, $u_{2}\in V$, $w\in W$,
$w'\in W'=\coprod_{n\in \mathbb{Z}}W_{(n)}^{*}$ ($W'$ is 
the dual $V$-module to $W$), the series 
\begin{eqnarray}
\label{porosyataw}
&&
 \langle w', Y_{W}(u_{1}, z_{1})Y_{W}(u_{2}, z_{2})w\rangle,
\nn
%%\quad 
&&
\langle w', Y_{W}(u_{2}, z_{2})Y_{W}(u_{1}, z_{1})w\rangle, %and  %%\quad\mbox{and}
\nn
&& 
\langle w', Y_{W}(Y_{V}(u_{1}, z_{1}-z_{2})u_{2}, z_{2})w\rangle, 
\end{eqnarray}
are absolutely convergent
in the regions $|z_{1}|>|z_{2}|>0$, $|z_{2}|>|z_{1}|>0$,
$|z_{2}|>|z_{1}-z_{2}|>0$, respectively, to a common rational function 
in $z_{1}$ and $z_{2}$ with the only possible poles at $z_{1}=0=z_{2}$ and 
$z_{1}=z_{2}$. 

The locality 
\[
Y_W(v_1, z_1) Y_W(v_2, z_2) \sim Y_W(v_2, z_2) Y_W(v_1, z_1),
\]
 and associativity 
\[
Y_W(v_1, z_1) Y_W(v_2, z_2) \sim Y_W(Y_V v_1, z_1- z_2) v_2, z_2),
\]
 properties for the vertex operators in a $V$-module $W$ follow from the
Jacobi identity \cite{K}.  
\item { $L_{W}(0)$-bracket formula}: For  $v\in V$,
\[
[L_{W}(0), Y_{W}(v, z)]=Y_{W}(L(0)v, z)+z\frac{d}{dz}Y_{W}(v, z).
\] 
\item { $L_W(0)$-grading property}: For $w\in W_{(\alpha)}$, there exists
$N\in \Z_{+}$ such that $(L_{W}(0)-\alpha)^{N}w=0$. 
\item { $L_W(-1)$-derivative property}: For $v\in V$,  
\[
\frac{d}{dz}Y_{W}(u, z)=Y_{W}(L_{V}(-1)u, z)=[L_{W}(-1), Y_{W}(u, z)].
\] 
\end{enumerate}
%%
%%%%%%%%%%%%%%%%%%%%%%%%%%%%%%%%%%%%%%%%%%%%%%%%%%%%%%%%%%%%%%%%%%%%%%%%%%%%%%%%%%%%%%%
%%%%%%%%%%%%%%%%%%%%%%%%%%%%%%%%%%%%%%%%%%%%%%%%%%%%%%%%%%%%%%%%%%%%%%%%%%%%%%%%%%%%%%%
%%
For $v\in V$, and $w \in W$, the intertwining operator 
\begin{eqnarray}
\label{interop}
&& Y_{WV}^{W}: V\to W,  
\nn
&&
v   \mapsto  Y_{WV}^{W}(w, z) v,    
\end{eqnarray}
 is defined by 
\begin{eqnarray}
\label{wprop}
Y_{WV}^{W}(w, z) v= e^{zL_W(-1)} Y_{W}(v, -z) w. 
\end{eqnarray}
%%
%%%%%%%%%%%%%%%%%%%%%%%%%%%%%%%%%%%%%%%%%%%%%%%%%%%%%%%%%%%%%%%%%%%%%%%%%%%%%%%%%%%%%%%%%%%%%%%%%%
%%%%%%%%%%%%%%%%%%%%%%%%%%%%%%%%%%%%%%%%%%%%%%%%%%%%%%%%%%%%%%%%%%%%%%%%%%%%%%%%%%%%%%%%%%%%%%%%%%
%%%
\subsection{Non-degenerate  invariant bilinear form on $V$} %The Li--Zamolodchikov (Li--Z) Metric}
\label{liza}
The subalgebra 
\[
\{L_V(-1),L_V(0),L_V(1)\}\cong SL(2,\mathbb{C}), 
\]
 associated with M\"{o}bius transformations on  
$z$ naturally acts on $V$,   (cf., e.g. \cite{K}). 
In particular, 
\begin{equation}
\gamma_{\lambda}=\left(
\begin{array}{cc}
0 & \lambda\\
-\lambda & 0\\	
\end{array}
\right)
:z\mapsto w=-\frac{\lambda^{2}}{z},
 \label{eq: gam_lam}
\end{equation}
is generated by 
\[
T_{\lambda }= \exp\left(\lambda L_V{(-1)}\right) 
%5
\; \exp\left({\lambda}^{-1}L_V(1)\right) \; \exp\left(\lambda L_V(-1)\right),   
\]
 where
\begin{equation}
T_{\lambda }Y(u,z)T_{\lambda }^{-1}=
Y\left(\exp \left(-\frac{z}{\lambda^{2}}L_V(1)\right)
\left(-\frac{z}{\lambda}\right)^{-2L_V(0)}u,-\frac{\lambda^{2}}{z}\right).  \label{eq: Y_U}
\end{equation}
In our considerations (cf. Appendix \ref{sphere}) of Riemann sphere %surface 
sewing, we use in particular, 
%%
%Later we will be particularly interested in 
%%
the M\"{o}bius map 
\[
z\mapsto z'= \epsilon/z,
\] 
 associated with the sewing condition \eqref{pinch} with 
\begin{equation}
\lambda=-\xi\epsilon^{\frac{1}{2}},
\label{eq:lamb_eps}
\end{equation}  
with $\xi\in\{\pm \sqrt{-1}\}$. % as previously introduced in \eqref{dz1dz2}.
%%
%For $u\in V$ of half-integral weight the action of $-\gamma_{\lambda}=\gamma_{-\lambda}$ is distinguished from 
%that of $\gamma_{\lambda}$ whereas for integral weight they are equivalent. In particular we must distinguish the choices 
%  $\lambda =\pm\sqrt{-1}$ in (\ref{eq: gam_lam}) corresponding to the
%inversion map $z\mapsto z^{-1}$ normally used to define the adjoint vertex operator. 
%%
The adjoint vertex operator \cite{K, FHL}  
is defined by %we %therefore 
%%
%define   
%%
\begin{equation}
Y^{\dagger }(u,z)=\sum_{n\in \Z}u^{\dagger }(n)z^{-n-1}= T_{\lambda}Y(u,z)T_{\lambda}^{-1}. \label{eq: adj op}
\end{equation}%
%%
%%%%%%%%%%%%%%%%%%%%%%%%%%%%%%%%%%%%%%%%%%%%%%%%%%%%%%5
%
A bilinear form $\langle . , . \rangle_{\lambda}$ on $V$ is 
invariant if for all $a,b,u\in V$, %\cite{Sche}, 
%5
if  
\begin{equation}
\langle Y(u,z)a,b\rangle_{\lambda} =%(-1)^{p(u)p(a)}
\langle a,Y^{\dagger }(u,z)b\rangle_{\lambda}, 
\label{eq: inv bil form}
\end{equation}%
i.e.
\[
 \langle u(n)a,b\rangle_{\lambda} =%(-1)^{p(u)p(a)}
\langle a,u^{\dagger }(n)b\rangle_{\lambda}.
\] 
Thus it follows that 
\begin{equation}
\label{dubay}
\langle L_V(0)a,b\rangle_{\lambda} =\langle a,L_V(0)b\rangle_{\lambda}, 
\end{equation}
 so that 
\begin{equation}
\label{condip}
\langle a,b\rangle_{\lambda} =0, 
\end{equation}
  if $wt(a)\not=wt(b)$ for homogeneous $a,b$.
 One also finds 
\[
\langle a,b\rangle_{\lambda} = \langle b,a \rangle_{\lambda}.  
\]
%$
 %\cite{FHL}. % , \cite{Sche}. 
%%
The form 
$\langle . , .\rangle_{\lambda}$ is unique up to normalization if $L_V(1)V_{1}=V_{0}$.  
%%
%(we choose the normalization $\langle \mathbf{1} ,\mathbf{1}\rangle_{\lambda}=1$ throughout) 
%and is non-degenerate if and only if $V$ is simple \cite{L}. 
%%
% We call such a unique non-degenerate 
%symmetric bilinear form the Li--Zamolodchikov (Li--Z) metric.
%%
 Given any $V$ basis $\{ u^{\alpha}\}$ we define the %Li--Z 
dual $V$ basis $\{ \overline{u}^{\beta}\}$ where 
\[
\langle u^{\alpha} ,\overline{u}^{\beta}\rangle_{\lambda}=\delta^{\alpha\beta}.
\] 
%%

%%%%%%%%%%%%%%%%%%%%%%%%%%%%%%%%%%%%%%%%%%%%%%%%%%%%%%%%%%%%%%%%%%%%%%%%%%%%%%%%%%%%%%%%%%%%%%%%%%
%%%%%%%%%%%%%%%%%%%%%%%%%%%%%%%%%%%%%%%%%%%%%%%%%%%%%%%%%%%%%%%%%%%%%%%%%%%%%%%%%%%%%%%%%%%%%%%%%%
%%%%%%%%%%%%%%%%%%%%%%%%%%%%%%%%%%%%%%%%%%%%%%%%%%%%%%%%%%%%%%%%%%%%%%%%%%%%%%%%%%%%%%%%%%%%%%%%%%%%%%%%%%%
\section{Appendix: %Geometrical realization of $\W$-spaces product on 
A sphere formed from sewing of two spheres}
\label{sphere}
%%

%%%%%%%%%%%%%%%%%%%%%%%%%%%%%%%%%%%%%%%%%%%%%%%%%%%%%%%%%%%%%%%%%%%%%%%%%%%%%%%%%%%%%%%%%%%%%%%%%%%%%%%%%%%%%%%%%%%%%%%
%Definition \ref{duplodef} is the general definition only, and one needs a concrete realizaion of the map \eqref{gendef}
%%
%which is given below using a suitable geometrical model.
%%
The matrix element for a number of vertex operators of a vertex algebra is usually associated \cite{FHL, FMS, TUY} 
with a vertex algebra character on a sphere. We extrapolate this notion to the case of $\W_{z_1, \ldots, z_n}$ spaces. 
In Section \ref{productc} we explained that a space $\W_{z_1, \ldots, z_n}$ can be associated with a Riemann sphere with marked points, 
while the product of two such spaces is then associated with a sewing of such two spheres with a number 
of marked % three puncured 
points 
and extra points with local coordinates identified with formal parameters of $\W_{x_1, \ldots, x_k}$ and $\W_{y_1, \ldots, y_n}$. 
In order to supply an appropriate geometric construction for the product,  
 we use the $\epsilon$-sewing procedure (described in this Appendix) for two initial spheres to obtain a matrix element associated with \eqref{gendef}. 

%%%%%%%%%%%%%%%%%%%%%%%%%%%%%%%%%%%%%%%%%%%%%%%%%%%%%%%%%%%%%%%%%%%%%%%%%%%%%%%%%%%%%%%%%%%%%%%%%%%%%%%%%%
\begin{remark}
In addition to the $\epsilon$-sewing procedure of two initial spheres, one can alternatively use 
the self-sewing procedure \cite{Y} for the sphere to get, at first, the torus, and then by sending parameters 
to appropriate limit by shrinking genus to zero. As a result, one obtains again the sphere but with a  
different parameterization.  In the case of spheres, such a procedure 
%%
%(see explicite formulation in Appendix \ref{rhorho}) 
%%leads to a similar 
%%
consideration of the product of $\W$-spaces so we focus in this paper on the $\epsilon$-formalizm only.  
\end{remark}
%%
%%%%%%%%%%%%%%%%%%%%%%%%%%%%%%%%%%%%%%%%%%%%%%%%%%%%%%%%%%%%%%%%%%%%%%%%%%%%%%%%%
%%%%%%%%%%%%%%%%%%%%%%%%%%%%%%%%%%%%%%%%%%%%%%%%%%%%%%%%%%%%%%%%%%%%%%%%%%%%%%%%%
%\subsection{Geometrical realization of $\W$-spaces product}
%%
%\label{sfery}
%%
In our particular case of $\W$-values rational functions obtained from matrix elements \eqref{def}  
two initial %auxillary  
auxiliary
 spaces we take Riemann spheres $\Sigma^{(0)}_a$, $a=1$, $2$, and the resulting 
space is formed by 
the sphere $\Sigma^{(0)}$ obtained by the procedure of sewing $\Sigma^{(0)}_a$. 
The formal parameters $(x_1, \ldots, x_k)$ and $(y_{1}, \ldots, y_n)$ are identified with 
local coordinates of $k$ and $n$ points on two initial spheres $\Sigma^{(0)}_a$, $a=1$, $2$ correspondingly.  
In the $\epsilon$ sewing procedure, some $r$ points 
%%parameters 
%%
among %$(x_1, \ldots, x_k)$ 
$(p_1, \ldots, p_k)$ 
may coincide with 
points 
among $(p'_{1}, \ldots, p'_n)$ %$(y_{k+1}, \ldots, y_n)$ 
when we identify the annuluses \eqref{zhopki}.   
This corresponds to the singular case of coincidence of $r$ formal parameters. 
%%
%%%%%%%%%%%%%%%%%%%%%%%%%%%%%%%%%%%%%%%%%%%%%%%%%%%%%%%%%%%%%%%%%%%%%%%%%%%%%
%%%%%%%%%%%%%%%%%%%%%%%%%%%%%%%%%%%%%%%%%%%%%%%%%%%%%%%%%%%%%%%%%%%%%%%%%%%%%
%%%%%%%%%%%%%%%%%%%%%%%%%%%%%%%%%%%%%%%%%%%%%%%%%%%%%%%%%%%%%%%%%%%%%%%
%%
%% 
%%
%In the $\epsilon $-sewing scheme we sew  two sphere $\Sigma_{a}^{(0)}$, $a=1$, $2$ with modular parameters 
%%
%... $\tau_{a}$ 
%%
%via the sewing relation (\ref{pinch}) to get the resulting sphere $\mathcal S$. 
%%
%Similarly to ref.~\cite{MT1} for VOAs, 
%%%%%%%%%%%%%%%%%%%%%%%%%%%%%%%%%%%%%%%%%%%%%%%%%%%%%%%%%%%%%%%%%%%%%%%%%%%%%%%%%%%%%%%%%%%%%%%%

%%%%%%%%%%%%%%%%%%%%%%%%%%%%%%%%%%%%%%%%%%%%%%%%%%%%%%%%%%%%%%%%%%%%%%%%%%%%%%
%%
Consider the sphere formed by sewing together two initial spheres in the sewing scheme referred to 
as the $\epsilon$-formalism in \cite{Y}. 
%%
%% ref\cite{MT1}, \cite{MT2}, \cite{TZ1}. 
%%
Let $\Sigma_a^{(0)}$,  %=\mathbb{C}/{\Lambda}_{a}$ for 
$a=1$, $2$ 
be %oriented tori with lattice 
to initial spheres.  
%% 
%${\Lambda}_{a}=2\pi i(\mathbb{Z}\tau _{a}\oplus \mathbb{Z})$ for 
%$\tau _{a}\in \mathbb{H}$.  
%%
%%
Introduce a complex sewing
parameter $\epsilon$ where 
\[
|\epsilon |\leq r_{1}r_{2},
\]
%%
%%%%%%%%%%%%%%%%%%%%%%%%%%%%%%%%%%%%%%%%%%%%%%%%%%%%%%%%%%%%%%%%%%%%%%%%%%%%%%%%%%%%%%%%%%
Consider $k$ distinct points on $p_i \in  \Sigma_{1}^{(0)}$, $i=1, \ldots, k$, 
with local coordinates $(x_1, \ldots, x_{k}) \in F_{k}\C$,  
and distinct points $p_j \in  \Sigma_{2}^{(0)}$, $j=1, \ldots, n$,
with local coordinates $(y_{1},\ldots ,y_{n})\in F_{n}\C$,  
%%
%For $u$, $v_1, \ldots, v_n \in V$, $\zeta_a \in \C$ given by \eqref{disk}, 
%%
with
\[
\left\vert x_{i}\right\vert
\geq |\epsilon |/r_{2}, 
\]
%%
 %and
%
\[
\left\vert y_{i}\right\vert \geq |\epsilon |/r_{1}. 
\] 
%%%%%%%%%%%%%%%%%%%%%%%%%%%%%%%%%%%%%%%%%%%%%%%%%%%%%%%%%%%%%%%%%%%%%%%%
%% 
Choose a local coordinate $z_{a}\in \mathbb{C}$ %/{\Lambda}_{a}$ 
on $\Sigma^{(0)}_a$ in the
neighborhood of points $p_{a}\in\Sigma^{(0)}_a$, $a=1$, $2$. 
%%
 %and 
Consider the closed disks 
\[
\left\vert \zeta_{a} \right\vert \leq r_{a}, 
\]
%%
%for $r_{a}<\frac{1}{2}D(q_{a})$ where \cite{Y} %\cite{MT2} 
%%
%\begin{equation*}
%%
%D(q_{a})=\min_{\lambda \in {\Lambda}_{a}, \lambda \neq 0}|\lambda |,
%%
%\end{equation*}
%%
%is ... %% the minimal lattice distance.  
%%
%%
 and excise the disk 
\begin{equation}
\label{disk}
\{
\zeta_{a}, \; \left\vert \zeta_{a}\right\vert \leq |\epsilon |/r_{\overline{a}}\}\subset 
\Sigma^{(0)}_a, 
\end{equation}
to form a punctured sphere  
\begin{equation*}
\widehat{\Sigma}^{(0)}_a=\Sigma^{(0)}_a \backslash \{\zeta_{a},\left\vert 
\zeta_{a}\right\vert \leq |\epsilon |/r_{\overline{a}}\}.
\end{equation*}
%
%%Here and below, 
%%
We use the convention 
\begin{equation}
\overline{1}=2,\quad \overline{2}=1.  
\label{bardef}
\end{equation}
Define the annulus
\begin{equation}
\label{zhopki}
\mathcal{A}_{a}=\left\{\zeta_{a},|\epsilon |/r_{\overline{a}}\leq \left\vert
\zeta_{a}\right\vert \leq r_{a}\right\}\subset \widehat{\Sigma}^{(0)}_a,
\end{equation}
and identify $\mathcal{A}_{1}$ and $\mathcal{A}_{2}$ as a single region 
$\mathcal{A}=\mathcal{A}_{1}\simeq \mathcal{A}_{2}$ via the sewing relation 
\begin{equation}
\zeta_{1}\zeta_{2}=\epsilon.   \label{pinch}
\end{equation}
In this way we obtain a genus zero compact Riemann surface 
\[
\Sigma^{(0)}=\left\{ \widehat {\Sigma}^{(0)}_1
\backslash \mathcal{A}_{1} \right\}
\cup \left\{\widehat{\Sigma}^{(0)}_2 \backslash 
\mathcal{A}_{2}\right\}\cup \mathcal{A}.
\] 
%%
%%
%%%%%%%%%%%%%%%%%%%%%%%%%%%%%%%%%%%%%%%%%%%%%%%%%%%%%%%%%%%%%%%%%%%%%%%%%%%%%%%%%%%%
%As a result of this geometric procedure we obtain the sphere with the module space 
%defined by \eqref{} and depending on $\epsilon$.  
%%
This sphere form a suitable geometrical model for the construction of a product of $\W$-valued rational forms
in Section \ref{productc}. 
%%
%%%%%%%%%%%%%%%%%%%%%%%%%%%%%%%%%%%%%%%%%%%%%%%%%%%%%%%%%%%%%%%%%%%%%%%%%%%%%%%%%%%%%%%%%%%
%%
\section{Appendix: proof of Proposition \ref{tosya}}
\label{duda}
%%
%%%%%%%%%%%%%%%%%%%%%%%%%%%%%%%%%%%%%%%%%%%%%%%%%%%%%%%%%%%%%%%%%%%%%%%%%%%%%%%%%%%%%%%%%%%%%%5
\begin{proof}
%%
%% 
%%%%%%%%%%%%%%%%%%%%%%%%%%%%%%%%%%%%%%%%%%%%%%%%%%%%%%%%%%%%%%%%%%%%%%%%%%%%%%%%%%%%%%%%%%%%%%%%%%%%%%%%%%%%%%%
%%
%%%%%%%%%%%%%%%%%%%%%%%%%%%%%%%%%%%%%%%%%%%%%%%%%%%%%%%%%%%%%%%%%%%%%%%%%%%%%%%%%%%%%%%%%%%%%%%%%%%%%%%%%%%%%%%%
%%
%In order to reconstruct 
%%
%$\left( \left(\delta^n_m \Phi \right)\; \Psi\right)$ and  $\left(\Phi \;\delta^{n'}_{m'} \Psi\right)$  
%%
%of \eqref{leibniz}, we add certain vanishing extra terms to obtain  
%%
%%%%%%%%%%%%%%%%%%%%%%%%%%%%%%%%%%%%%%%%%%%%%%%%%%%%%%%%%%%%%%%%%%%%%%%%%%%%%%%%%%%%%%%%%%%%%%%%%%%%%
%%%%%%%%%%%%%%%%%%%%%%%%%%%%%%%%%%%%%%%%%%%%%%%%%%%%%%%%%%%%%%%%%%%%%%%%%%%%%%%%%%%%%%%%%%%%%%%%%%%%%
%%
%% 
%%%%%%%%%%%%%%%%%%%%%%%%%%%%%%%%%%%%%%%%%%%%%%%%%%%%%%%%%%%%%%%%%%%%%%%%%%%%%%%%%%%%%%%%%%%%%%%%%%%%%%%%%
%%
For a vertex operator $Y_{V, W}(v,z)$ let us introduce a notation $\omega_{V, W}=Y_{V, W}(v,z)\; dz^{{\rm wt} v}$. 
Let us use notations \eqref{zsto} and \eqref{notari}. 
According to \eqref{hatdelta}, the action of 
$\delta_{m + m'-t}^{k + n-r}$ on $\widehat{R} \F(
v_1, x_1; \ldots; v_k, x_k; v'_1, y_1; \ldots; v'_k, y_n
%%
%%v_1', z_; \ldots; v'_{k+n'-r}, z_{k+n'-r}
%%
; \epsilon)$
is given by 
%%
%%%%%%%%%%%%%%%%%%%%%%%%%%%%%%%%%%%%%%%%%%%%%%%%%%%%%%%%%%%%%%%%%%%%%%%%%%%%%%%%%%%%%%%%
%% 
%%%%%%%%%%%%%%%%%%%%%%%%%%%%%%%%%%%%%%%%%%%%%%%%%%%%%%%%%%%%%%%%%%%%%%%%%%%%%%%%%%%%%%%%
\begin{eqnarray*}
&&   
\langle w', 
 \delta_{m + m'-t}^{k + n-r} \widehat{R} \; \F( 
v_1, x_1; \ldots; v_k, x_k; v'_1, y_1; \ldots; v'_n, y_n
%v_1', z_; \ldots; v'_{k+n'-r}, z_{k+n'-r}
; \epsilon) \rangle 
\nn
%%
%%%%%%%%%%%%%%%%%%%%%%%%%%%%%%%%%%%%%%%%%%%%%%%%%%%%%%%%%%%%%%%%%%%%%%%%%%%%%%%%%%%%%%%%%%%%%%
&& \quad =
\langle w', \sum_{i=1}^{k%+n
}(-1)^{i} \; 
\widehat{R} \; \F ( \widetilde{v}_1, z_1; \ldots;  \widetilde{v}_{i-1}, z_{i-1}; \;  \omega_V (\widetilde{v}_i, z_i  
 - z_{i+1}) 
 \widetilde{v}_{i+1}, z_{i+1}; \; \widetilde{v}_{i+2}, z_{i+2}; 
\nn
&& \qquad \qquad \qquad 
\ldots;  \widetilde{v}_k, z_k;  \widetilde{v}_{k+1}, z_{k+1}; \ldots; \widetilde{v}_{k+n}, z_{k+n}; \epsilon ) \rangle   
%%
%\nn
%&&
%\nn%
\end{eqnarray*}
%%
%%%%%%%%%%%%%%%%%%%%%%%%%%%%%%%%%%%%%%%%%%%%%%%%%%%%%%%%%%%%%%%%%%%%%%%%%%
\begin{eqnarray*}
%%%%%%%%%%%%%%%%
&& \qquad + 
\sum_{i=1}^{n-r}(-1)^{i} \; \langle w', 
\F \left( \widetilde{v}_1, z_1; \ldots; \widetilde{v}_k, z_k;  
\widetilde{v}_{k+1}, z_{k+1}; \ldots; \widetilde{v}_{k+i-1}, z_{k+i-1}; \right. 
\nn
&&\qquad \qquad \qquad  
  \omega_V \left(\widetilde{v}_{k+i}, z_{k+i}  
 - z_{k+i+1} ) \; 
 \widetilde{v}_{k+i+1}, z_{k+i+1}; \right.
\nn
&&
\left. 
\qquad \qquad \qquad \qquad \qquad \qquad 
\widetilde{v}_{k+i+2}, z_{k+i+2}; \ldots; \widetilde{v}_{k+n-r}, z_{k+n-r}; \epsilon \right) \rangle  
%%
%\nn
%%
\end{eqnarray*}
%%
%%%%%%%%%%%%%%%%%%%%%%%%%%%%%%%%%%%%%%%%%%%%%%%%%%%%%%%%%%%%%%%%%%%%%%%%%%
\begin{eqnarray*}
&& \qquad +  \langle w', 
 \omega_W \left(\widetilde{v}_1, z_1   
  \right) \; \F (\widetilde{v}_2, z_2
; \ldots; \widetilde{v}_{k}, z_k; \widetilde{v}_{k+1}, z_{k+1}; \ldots; \widetilde{v}_{k+n-r}, z_{k+n-r}; \epsilon   
)  \rangle  
%\nn
%%
%\nonumber
\end{eqnarray*}
%%
%%%%%%%%%%%%%%%%%%%%%%%%%%%%%%%%%%%%%%%%%%%%%%%%%%%%%%%%%%%%%%%%%%%%%%%%%%
\begin{eqnarray*}
 & &\qquad +  \langle w, (-1)^{k+n+1-r}    
 \omega_W(\widetilde{v}_{k+n-r+1}, z_{k+n-r+1}  
) 
\;
\nn
&&
 \qquad \qquad \qquad \F(\widetilde{v}_1, z_1
; \ldots; \widetilde{v}_k, z_k; \widetilde{v}_{k+1}, z_{k+1}; \ldots; \widetilde{v}_{k+n-r}, z_{k+n-r}; \epsilon )
 \rangle 
%%
%%\nn
%&&
%\nn
%%
\end{eqnarray*}
%%
%%%%%%%%%%%%%%%%%%%%%%%%%%%%%%%%%%%%%%%%%%%%%%%%%%%%%%%%%%%%%%%%%%%%%%%%%%
\begin{eqnarray*}
%%
%%%%%%%%%%%%%%%%%%%%%%%%%%%%%%%%%%%%%%%%%%%%%%%%%%%%%%%%%%%%%%%%%%%%%%%%
%%%%%%%%%%%%%%%%%%%%%%%%%%%%%%%%%%%%%%%%%%%%%%%%%%%%%%%%%%%%%%%%%%%%%%%%
&& \quad = 
\sum\limits_{u\in V} 
\langle w', \sum_{i=1}^{k}(-1)^{i} \; Y^W_{VW}(  
\F ( \widetilde{v}_1, z_1; \ldots;  \widetilde{v}_{i-1}, z_{i-1}; \; \omega_V (\widetilde{v}_i, z_i  
 - z_{i+1}) 
 \widetilde{v}_{i+1}, z_{i+1}; \; 
\nn
&& \qquad \qquad \qquad  
\widetilde{v}_{i+2}, z_{i+2};  \ldots;  \widetilde{v}_k, z_k), \zeta_1) u \rangle 
\nn
&&
\qquad \qquad \qquad \qquad \qquad \qquad
 \langle w', Y^W_{VW}( \F(\widetilde{v}_{k+1}, z_{k+1}; \ldots; \widetilde{v}_{k+n-r}, z_{k+n-r}), \zeta_2) 
\overline{u} \rangle  
%%
%%%%%%%%%%%%%%%%%%%%%%%%%%%%%%%%%%%%%%%%%%%%%%%%%%%%%%%%%%%%%%%%%%%%%%%
%\nn
\end{eqnarray*}
%%
%%%%%%%%%%%%%%%%%%%%%%%%%%%%%%%%%%%%%%%%%%%%%%%%%%%%%%%%%%%%%%%%%%%%%%%%%%
\begin{eqnarray*}
&& \qquad + 
\sum\limits_{u\in V}  \sum_{i=1}^{n-r}(-1)^{i} \; \langle w',  Y^W_{VW}(
\F \left( \widetilde{v}_1, z_1; \ldots; \widetilde{v}_k, z_k), \zeta_1) u \rangle \right. 
\nn
&&
\qquad \qquad \qquad 
   \langle w',  
Y^W_{VW}( \F(\widetilde{v}_{k+1}, z_{k+1}; \ldots; \widetilde{v}_{k+i-1}, z_{k+i-1};  
\nn
&& 
\qquad \qquad \qquad  \qquad \omega_V ( \widetilde{v}_i, z_{k+i}  
 - z_{k+i+1}) \; 
 \widetilde{v}_{k+i+1}, z_{k+i+1}; \widetilde{v}_{k+i+2}, z_{k+i+2}; 
\nn
&&
\qquad \qquad \qquad  \qquad \qquad  \ldots; \widetilde{v}_{k+n-r}, z_{k+n-r} ), \zeta_2) \overline{u} \rangle  
%%
%%
%\nn
%&&
%\nn
\end{eqnarray*}
%%
%%%%%%%%%%%%%%%%%%%%%%%%%%%%%%%%%%%%%%%%%%%%%%%%%%%%%%%%%%%%%%%%%%%%%%%%%%
\begin{eqnarray*}
%%%%%%%%%%%%%%%%%%%%%%%%%%%%%%%%%%%%%%%%%%%%%%%%%%%%%%%%%%%%%%%%%%%%%%%%%%%%
%%
&& \qquad +  \sum\limits_{u\in V} \langle w', Y^W_{VW}(
 \omega_W \left(\widetilde{v}_1, z_1 \right) \; \F (\widetilde{v}_2, z_2 ; \ldots; \widetilde{v}_{k}, z_k), \zeta_1) u \rangle 
\nn
&&
 \qquad \qquad \langle w', Y^W_{VW}( \F( \widetilde{v}_{k+1}, z_{k+1}; \ldots; \widetilde{v}_{k+n-r}, z_{k+n-r} ), \zeta_2) \overline{u} \rangle   
%\nn
\end{eqnarray*}
%%
%%%%%%%%%%%%%%%%%%%%%%%%%%%%%%%%%%%%%%%%%%%%%%%%%%%%%%%%%%%%%%%%%%%%%%%%%%
\begin{eqnarray*}
&&
\qquad +  \sum\limits_{u\in V}  \langle w', Y^W_{VW}( (-1)^{k+1}  
 \omega_W \left(\widetilde{v}_{k+1}, z_{k+1} 
  \right) \;  \F (\widetilde{v}_1, z_1 ; \ldots; \widetilde{v}_{k}, z_k), \zeta_1) u\rangle  
\nn
&& \qquad \qquad \qquad 
\langle w', Y^W_{VW}(
\F( \widetilde{v}_{k+2}, z_{k+2}; \ldots; \widetilde{v}_{k+n-r}, z_{k+n-r}), \zeta_2) \overline{u} \rangle  
%%
%\nn
\end{eqnarray*}
%%
%%%%%%%%%%%%%%%%%%%%%%%%%%%%%%%%%%%%%%%%%%%%%%%%%%%%%%%%%%%%%%%%%%%%%%%%%%
\begin{eqnarray*}
&&
\qquad - 
%%%%%%%%%%%%%%%%%%%%%%%%%%%%%%%%%%%%%%%%%%%%%%%%%%%%%5
\sum\limits_{u\in V} \langle w', (-1)^{k+1}   \langle w',  Y^W_{VW}(  
 \omega_W \left(\widetilde{v}_{k+1}, z_{k+1} 
  \right) \; \F (\widetilde{v}_1, z_1 ; \ldots; \widetilde{v}_{k}, z_k), \zeta_1) u \rangle  
\nn
&&
\qquad \qquad  \qquad 
\langle w',  Y^W_{VW}( \F(\widetilde{v}_{k+2}, z_{k+2}; \ldots; \widetilde{v}_{k+n-r}, z_{k+n-r}),  \zeta_2)  \overline{u} 
\rangle  
%%
%%%%%%%%%%%%%%%%%%%%%%%%%%%%%%%%%%%%%%%%%%%%%%%%%%%%%%%
%%
%%%%%%%%%%%%%%%%%%%%%%%%%%%%%%%%%%%%%%%%%%%%%%%%%%%%%%5
%\nn
%&&
\nn
&&
\qquad +  \sum\limits_{u\in V}  \langle w', Y^W_{VW}(  
 \F(\widetilde{v}_1, z_1; \ldots; \widetilde{v}_k, z_k), \zeta_1)  u\rangle  
\nn
&&
\qquad \qquad 
 \langle w', Y^W_{VW}(  
\omega_W(\widetilde{v}_{k+n-r+1}, z_{k+n-r+1})\; 
\nn
&&
 \qquad \qquad \qquad  \qquad  \F(\widetilde{v}_{k+1}, z_{k+1}; \ldots; \widetilde{v}_{k+n-r}, z_{k+n-r}  ), \zeta_2)  \overline{u}\rangle 
%%
%%%%
\nn
&& \qquad 
%%%%%%%%5%%%%%%%%%%%%%%%%%%%%%%%%%%%%%%%%%%%%%%%%%%%%%%%%%%%%%%%
- \sum\limits_{u\in V} \langle w', Y^W_{VW}(  
 \F(\widetilde{v}_1, z_1; \ldots; \widetilde{v}_k, z_k), \zeta_1) \rangle  
\nn
&& 
\qquad \qquad \langle w',  Y^W_{VW}(
   \omega_W(\widetilde{v}_{k+n-r+1}, z_{k+n-r+1})
\nn
&&
 \qquad \qquad \qquad  \qquad  \F( \widetilde{v}_{k+1}, z_{k+1}; \ldots; \widetilde{v}_{k+n-r}, z_{k+n-r}  ), \zeta_2) \rangle 
%%
%\nn
%&&
%\nn
%%
\end{eqnarray*}
%%
%%%%
%%%%%%%%%%%%%%%%%%%%%%%%%%%%%%%%%%%%%%%%%%%%%%%%%%%%%%%%%%%%%%%%%%%%%%%%%%
\begin{eqnarray*}
%%
%%%%%%%%%%%%%%%%%%%%%%%%%%%%%%%%%%%%%%%%%%%%%%%%%%%%%%%%%%%%%%%%%%%%%%%%%%
%%%%%%%%%%%%%%%%%%%%%%%%%%%%%%%%%%%%%%%%%%%%%%%%%%%%%%%%%%%%%%%%%%%%%%%%%%
&& \quad = 
\sum\limits_{u\in V} 
\langle w',  \; Y^W_{VW} (  
%%%%
\delta^k_m\F ( \widetilde{v}_1, z_1; \ldots;  \widetilde{v}_k, z_k), \zeta_1 ) u \rangle 
%%%%
\nn
&& 
\qquad \qquad \qquad \qquad 
\langle w', Y^W_{VW}( \F(\widetilde{v}_{k+1}, z_{k+1}; \ldots; \widetilde{v}_{k+n-r}, z_{k+n-r}), \zeta_2) 
\overline{u} \rangle  
%%%%
%%%%%%%%%%%%%%%%%%%%%%%%%%%%%%%%%%%%%%%%%%%%%%%%%%%%%%%%%%%%%%%%%%%%%%%
\nn
&& \qquad + (-1)^k 
\sum\limits_{u\in V}   \langle w',  Y^W_{VW}(
\F ( \widetilde{v}_1, z_1; \ldots; \widetilde{v}_k, z_k), \zeta_1) u \rangle   
\nn
&&
\qquad \qquad \qquad 
  \langle w',  Y^W_{VW}( \delta^{n-r}_{m'-t}  
\F(\widetilde{v}_{k+1}, z_{k+1}; \ldots;  \widetilde{v}_{k+n-r}, z_{k+n-r} ), \zeta_2 ) \overline{u} \rangle  
%%
%%
%\nn
%&&
%\nn
%%
%&& 
%%
\end{eqnarray*}
%%
%%%%%%%%%%%%%%%%%%%%%%%%%%%%%%%%%%%%%%%%%%%%%%%%%%%%%%%%%%%%%%%%%%%%%%%%%%
\begin{eqnarray*}
&& \quad = 
\langle w',   
\delta^k_m\F ( \widetilde{v}_1, z_1; \ldots;  \widetilde{v}_k, z_k) \cdot_\epsilon 
%%
%\nn
%&& 
%\qquad \qquad \qquad \qquad 
\langle w',  \F(\widetilde{v}_{k+1}, z_{k+1}; \ldots; \widetilde{v}_{k+n-r}, z_{k+n-r}) \rangle   
%%
%%%%%%%%%%%%%%%%%%%%%%%%%%%%%%%%%%%%%%%%%%%%%%%%%%%%%%%%%%%%%%%%%%%%%%%
\nn
&& \qquad + (-1)^k
%%
%\sum\limits_{u\in V}
%%
   \langle w',  
 \F ( \widetilde{v}_1, z_1; \ldots; \widetilde{v}_k, z_k) \cdot_{\epsilon}  
     \delta^{n-r}_{m'-t}  \F(\widetilde{v}_{k+1}, z_{k+1}; \ldots;  \widetilde{v}_{k+n-r}, z_{k+n-r} ) \rangle,  
\end{eqnarray*}
%%
%%%%%%%%%%%%%%%%%%%%%%%%%%%%%%%%%%%%%%%%%%%%%%%%%%%%%%%%%%%%%%%%%%%%%%%%%%%5
%%%%%%%%%%%%%%%%%%%%%%%%%%%%%%%%%%%%%%%%%%%%%%%%%%%%%%%%%%%%%%%%%%%%%%%%%%%%a
since, 
\begin{eqnarray*}
&& \sum\limits_{u\in V} \langle w', (-1)^{k+1}   Y^W_{VW}(  
 \omega_W \left(\widetilde{v}_{k+1}, z_{k+1} 
  \right) \; \F (\widetilde{v}_1, z_1 ; \ldots; \widetilde{v}_{k}, z_k), \zeta_1) u \rangle  
\nn
&&
\qquad \qquad \langle w',  Y^W_{VW}( \F(\widetilde{v}_{k+2}, z_{k+2}; \ldots; \widetilde{v}_{k+n-r}, z_{k+n-r}),  \zeta_2) \overline{u} 
\rangle  
%\nn
%&&
\end{eqnarray*}
%%%%%%%%%%%%%%%%%%%%%%%%%%%%%%%%%%%%%%%%%%%%%%%%%%%%%%%%%%%%%%%%%%%%%%%%%%%%%%%%%%%%%%
\begin{eqnarray*} 
%\nn
&&
=\sum\limits_{u\in V} \langle w', (-1)^{k+1} e^{\zeta_1 L_W{(-1)}}   Y_W(u, -\zeta_1)  \;  
 \omega_W \left(\widetilde{v}_{k+1}, z_{k+1}  
  \right) \; \F (\widetilde{v}_1, z_1 ; \ldots; \widetilde{v}_{k}, z_k) \rangle  
\nn
&&
\qquad \qquad \langle w',  Y^W_{VW}( \F(\widetilde{v}_{k+2}, z_{k+2}; \ldots; \widetilde{v}_{k+n-r}, z_{k+n-r}),  \zeta_2) \overline{u}
\rangle  
%%
%\nn
%&&
%%%%%%%%%%%%%%%%%%%%%%%%%%%%%%%%%%%%%%%%%%%%%%%%%%%%%%%%%%%%%%%%%%%%%%%%%%%%%%%%%%%%%%
%\nn
\end{eqnarray*}
%%%%%%%%%%%%%%%%%%%%%%%%%%%%%%%%%%%%%%%%%%%%%%%%%%%%%%%%%%%%%%%%%%%%%%%%%%%%%%%%%%%%%%
\begin{eqnarray*} 
&&
=\sum\limits_{u\in V} \langle w', (-1)^{k+1} e^{\zeta_1 L_W{(-1)}}  \omega_W \left(\widetilde{v}_{k+1}, z_{k+1}  \right) 
 Y_W(u, -\zeta_1)  \;  
  \; \F (\widetilde{v}_1, z_1 ; \ldots; \widetilde{v}_{k}, z_k) \rangle  
\nn
&&
\qquad \qquad \langle w',  Y^W_{VW}( \F(\widetilde{v}_{k+2}, z_{k+2}; \ldots; \widetilde{v}_{k+n-r}, z_{k+n-r}),  \zeta_2) \overline{u}
\rangle  
%%
%\nn
%&&
\end{eqnarray*}
%%%%%%%%%%%%%%%%%%%%%%%%%%%%%%%%%%%%%%%%%%%%%%%%%%%%%%%%%%%%%%%%%%%%%%%%%%%%%%%%%%%%%%
\begin{eqnarray*} 
%%%%%%%%%%%%%%%%%%%%%%%%%%%%%%%%%%%%%%%%%%%%%%%%%%%%%%%%%%%%%%%%%%%%%%%%%%%%%%%%%%%%%%
\nn
&&
=%\sum\limits_{v\in V} 
\sum\limits_{u\in V} \langle w', (-1)^{k+1} \; 
\omega_W \left(\widetilde{v}_{k+1}, z_{k+1} +\zeta_1  \right)\;  e^{\zeta_1 L_W{(-1)}}  
 Y_W(u, -\zeta_1)  \;  
  \; \F (\widetilde{v}_1, z_1 ; \ldots; \widetilde{v}_{k}, z_k) \rangle  
\nn
&&
\qquad \qquad \langle w',  Y^W_{VW}( \F(\widetilde{v}_{k+2}, z_{k+2}; \ldots; \widetilde{v}_{k+n-r}, z_{k+n-r}),  \zeta_2) \overline{u}
\rangle  
%%
%\nn
%&&
\end{eqnarray*}
%%%%%%%%%%%%%%%%%%%%%%%%%%%%%%%%%%%%%%%%%%%%%%%%%%%%%%%%%%%%%%%%%%%%%%%%%%%%%%%%%%%%%%
\begin{eqnarray*} 
%%%%%%%%%%%%%%%%%%%%%%%%%%%%%%%%%%%%%%%%%%%%%%%%%%%%%%%%%%%%%%%%%%%%%%%%%%%%%%%%%%%%%%
%\nn
&&
=\sum\limits_{v\in V} 
\sum\limits_{u\in V} 
\langle v', (-1)^{k+1} \; \omega_W \left(\widetilde{v}_{k+1}, z_{k+1}+\zeta_1  \right) w \rangle  
\nn
&&
\qquad \qquad \langle w',  e^{\zeta_1 L_W{(-1)}}  
 Y_W(u, -\zeta_1)  \;  
  \; \F (\widetilde{v}_1, z_1 ; \ldots; \widetilde{v}_{k}, z_k) \rangle  
\nn
&&
\qquad \qquad \langle w',  Y^W_{VW}( \F(\widetilde{v}_{k+2}, z_{k+2}; \ldots; \widetilde{v}_{k+n-r}, z_{k+n-r}),  \zeta_2) \overline{u}
\rangle  
%%
%\nn
%&&
\end{eqnarray*}
%%%%%%%%%%%%%%%%%%%%%%%%%%%%%%%%%%%%%%%%%%%%%%%%%%%%%%%%%%%%%%%%%%%%%%%%%%%%%%%%%%%%%%
\begin{eqnarray*} 
%%%%%%%%%%%%%%%%%%%%%%%%%%%%%%%%%%%%%%%%%%%%%%%%%%%%%%%%%%%%%%%%%%%%%%%%%%%%%%%%%%%%%%
%\nn
&&
= 
\sum\limits_{u\in V} 
%%
%\nn
%&&
%%
%\qqaud \qquad 
 \langle w',  e^{\zeta_1 L_W{(-1)}}  
 Y_W(u, -\zeta_1)  \;  
  \; \F (\widetilde{v}_1, z_1 ; \ldots; \widetilde{v}_{k}, z_k) \rangle  
\nn
&&
\qquad \qquad \sum\limits_{v\in V} \langle v', (-1)^{k+1} \; \omega_W \left(\widetilde{v}_{k+1}, z_{k+1}+\zeta_1  \right) w \rangle   
\nn
&&
\langle w',  Y^W_{VW}( \F(\widetilde{v}_{k+2}, z_{k+2}; \ldots;
\nn
&&
 \qquad \qquad \qquad \qquad \widetilde{v}_{k+n-r}, z_{k+n-r}),  \zeta_2) \overline{u}
\rangle  
%%
%\nn
%&&
\end{eqnarray*}
%%%%%%%%%%%%%%%%%%%%%%%%%%%%%%%%%%%%%%%%%%%%%%%%%%%%%%%%%%%%%%%%%%%%%%%%%%%%%%%%%%%%%%
\begin{eqnarray*} 
%%%%%%%%%%%%%%%%%%%%%%%%%%%%%%%%%%%%%%%%%%%%%%%%%%%%%%%%%%%%%%%%%%%%%%%%%%%%%%%%%%%%%%
%\nn
&&
= 
\sum\limits_{u\in V} 
%%
%\nn
%&&
%%
%\qqaud \qquad 
 \langle w',  
 Y^{W}_{VW}(  
   \F (\widetilde{v}_1, z_1 ; \ldots; \widetilde{v}_{k}, z_k) , \zeta_1) u \; \rangle  
\nn
&&
\qquad \qquad 
 \langle w', (-1)^{k+1} \; \omega_W \left(\widetilde{v}_{k+1}, z_{k+1}+\zeta_1  \right) 
\;
\nn
&&
 \qquad \qquad  Y^W_{VW}( \F(\widetilde{v}_{k+2}, z_{k+2}; \ldots; \widetilde{v}_{k+n-r}, z_{k+n-r}),  \zeta_2) \overline{u}
\rangle  
%%
%\nn
%&&
%%%%%%%%%%%%%%%%%%%%%%%%%%%%%%%%%%%%%%%%%%%%%%%%%%%%%%%%%%%%%%%%%%%%%%%%%%%%%%%%%%%%%
%\nn
\end{eqnarray*}
%%%%%%%%%%%%%%%%%%%%%%%%%%%%%%%%%%%%%%%%%%%%%%%%%%%%%%%%%%%%%%%%%%%%%%%%%%%%%%%%%%%%%%
\begin{eqnarray*} 
&&
= 
\sum\limits_{u\in V} 
%%
%\nn
%&&
%%
%\qqaud \qquad 
 \langle w',  
 Y^{W}_{VW}(  
   \F (\widetilde{v}_1, z_1 ; \ldots; \widetilde{v}_{k}, z_k) , \zeta_1) u \; \rangle  
\nn
&&
\qquad \qquad 
 \langle w', (-1)^{k+1} \; \omega_W \left(\widetilde{v}_{k+1}, z_{k+1}+\zeta_1  \right) 
\nn
&&
\qquad \qquad 
\; e^{\zeta_2 L_W{(-1)}}    Y_W(\overline{u}, -\zeta_2) \; \F(\widetilde{v}_{k+2}, z_{k+2}; \ldots; \widetilde{v}_{k+n-r}, z_{k+n-r}) 
\rangle  
%%
%\nn
%&&
\end{eqnarray*}
%%%%%%%%%%%%%%%%%%%%%%%%%%%%%%%%%%%%%%%%%%%%%%%%%%%%%%%%%%%%%%%%%%%%%%%%%%%%%%%%%%%%%%
\begin{eqnarray*} 
%%%%%%%%%%%%%%%%%%%%%%%%%%%%%%%%%%%%%%%%%%%%%%%%%%%%%%%%%%%%%%%%%%%%%%%%%%%%%%%%%%%%%%
%\nn
&&
= 
\sum\limits_{u\in V} 
%%
%\nn
%&&
%%
%\qqaud \qquad 
 \langle w',  
 Y^{W}_{VW}(  
   \F (\widetilde{v}_1, z_1 ; \ldots; \widetilde{v}_{k}, z_k) , \zeta_1) u \; \rangle  
\nn
&&
\qquad 
 \langle w', (-1)^{k+1} \; 
\; e^{\zeta_2 L_W{(-1)}} \;   Y_W(\overline{u}, -\zeta_2)
\;
\omega_W \left(\widetilde{v}_{k+1}, z_{k+1}+\zeta_1-\zeta_2 \right)  
\nn
&&
\qquad \qquad \; \F(\widetilde{v}_{k+2}, z_{k+2}; \ldots; \widetilde{v}_{k+n-r}, z_{k+n-r}) 
\rangle  
%%
%\nn
%&&
%%
\end{eqnarray*}
%%%%%%%%%%%%%%%%%%%%%%%%%%%%%%%%%%%%%%%%%%%%%%%%%%%%%%%%%%%%%%%%%%%%%%%%%%%%%%%%%%%%%%
\begin{eqnarray*} 
%%%%%%%%%%%%%%%%%%%%%%%%%%%%%%%%%%%%%%%%%%%%%%%%%%%%%%%%%%%%%%%%%%%%%%%%%%%%%%%%%%%%%%
%\nn
&& 
\qquad \qquad 
= \sum\limits_{u\in V} \langle w', Y^W_{VW}(  
 \F(\widetilde{v}_1, z_1; \ldots; \widetilde{v}_k, z_k), \zeta_1) u \rangle  
\nn
&& 
\qquad \qquad \langle w',  Y^W_{VW}(
   \omega_W(\widetilde{v}_{k+1}, z_{k+1}) \; \F( \widetilde{v}_{k+2}, z_{k+2}; \ldots; \widetilde{v}_{k+n-r}, z_{k+n-r}  ), \zeta_2) \overline{u} \rangle,   
\end{eqnarray*}
%%
%%%%%%%%%%%%%%%%%%%%%%%%%%%%%%%%%%%%%%%%%%%%%%%%%%%%%%%%%%%%%%%%%%%%%%%%%%%%%%%%%%%%%%%%%%%%
%%
due to locality \eqref{porosyataw} of vertex opertors, and arbitrarness of $\widetilde{v}_{k+1}\in V$ and $z_{k+1}$, 
we can always put
\[
\omega_W \left(\widetilde{v}_{k+1}, z_{k+1}+\zeta_1-\zeta_2 \right)  =\omega_W(\widetilde{v}_{k+2}, z_{k+2}), 
\]
for $\widetilde{v}_{k+1}=\widetilde{v}_{k+2}$, $z_{k+2}= z_{k+1}+\zeta_2-\zeta_1$.  
%%
%we find \eqref{}, 
%%
\end{proof}
%%%%%%%%%%%%%%%%%%%%%%%%%%%%%%%%%%%%%%%%%%%%%%%%%%%%%%%%%%%%%%%%%%%%%%%%%%%%%%%%%%%%%%%%%

%%%%%%%%%%%%%%%%%%%%%%%%%%%%%%%%%%%%%%%%%%%%%%%%%%%%%%%%%%%%%%%%%%%%%%%%%%%%%%%%%%%%%%%%%%
%%%%%%%%%%%%%%%%%%%%%%%%%%%%%%%%%%%%%%%%%%%%%%%%%%%%%%%%%%%%%%%%%%%%%%%%%%%%%%%%%%%%%%%%%%
%%

\end{document}